\documentclass[a4paper,titlepage,11pt]{amsart}

\usepackage[foot]{amsaddr}
\usepackage[utf8]{inputenc}
\usepackage[T1]{fontenc}
\usepackage[english]{babel}
\usepackage{verbatim}
\usepackage[left=3.1cm,right=3.1cm,top=3.2cm,bottom=3.2cm]{geometry}
\usepackage{amsfonts,amssymb,amsthm,nccmath,mathtools}%\usepackage[intlimits]{amsmath}
\usepackage{braket}
\usepackage{enumerate}
\usepackage{hyperref}
\usepackage{cleveref}
\usepackage{bm}
\usepackage{tikz}\usetikzlibrary{cd}
\usepackage{adjustbox}
\usepackage{algorithm2e}

\newcommand\Round[1]{{\left(#1\right)}}

\newcommand\Z{\mathbb{Z}}

\newcommand\C{\mathbb{C}}

\newcommand\K{\mathbb{K}}
\newcommand\A{\mathbb{A}}

\newcommand\Def[1]{\textbf{{#1}}}

\newcommand\orb{\mathcal{O}}

\DeclareMathOperator{\GL}{GL}
\DeclareMathOperator{\Flag}{\textit{Fl}}
\DeclareMathOperator{\Gr}{Gr}
\DeclareMathOperator{\rep}{rep}

\DeclareMathOperator{\modCat}{mod}
\DeclareMathOperator{\op}{\bm{\cdot}}
\DeclareMathOperator{\im}{Im}
\DeclareMathOperator{\rk}{rk}
\DeclareMathOperator{\Hom}{Hom}
\DeclareMathOperator{\Ext}{Ext}

\DeclareMathOperator{\Aut}{Aut}
\DeclareMathOperator{\Stab}{Stab}
\DeclareMathOperator{\Mat}{Mat}
\DeclareMathOperator{\id}{id}
\DeclareMathOperator{\p}{p}

\DeclareMathOperator{\dimvec}{{\bf dim}}

\theoremstyle{plain}
\newtheorem{theorem}[equation]{Theorem}
\crefname{theorem}{\bf{Theorem}}{\bf{theorems}}
\newtheorem{lemma}[equation]{Lemma}
\crefname{lemma}{\bf{Lemma}}{\bf{lemmas}}
\newtheorem{proposition}[equation]{Proposition}
\crefname{proposition}{\bf{Proposition}}{\bf{propositions}}
\newtheorem{corollary}[equation]{Corollary}
\crefname{corollary}{\bf{Corollary}}{\bf{corollaries}}
\theoremstyle{remark}
\newtheorem{remark}[equation]{Remark}
\crefname{remark}{\textit{Remark}}{\textit{remarks}}
\theoremstyle{definition}
\newtheorem{definition}[equation]{Definition}
\crefname{definition}{\bf{Definition}}{\bf{definitions}}
\newtheorem{example}[equation]{Example}
\crefname{example}{\bf{Example}}{\bf{examples}}
\newtheorem{conjecture}[equation]{Conjecture}
\crefname{conjecture}{\bf{Conjecture}}{\bf{conjectures}}
\newtheorem*{problem*}{Future perspectives}
\crefname{problem}{\bf{Problem}}{\bf{problems}}

\theoremstyle{plain}
\newtheorem{maintheorem}{Theorem}
\crefname{maintheorem}{\bf{Theorem}}{\bf{theorems}}

\theoremstyle{plain}
\newtheorem{mainconjecture}{Conjecture}
\crefname{mainconjecture}{\bf{Conjecture}}{\bf{conjectures}}

\definecolor{forestgreen}{rgb}{0.13, 0.55, 0.13}


\title[Linear degenerations of Schubert varieties]{Linear degenerations of Schubert varieties}

\numberwithin{equation}{section}

\DontPrintSemicolon
\SetKwInOut{Input}{input}
\SetKwInOut{Output}{output}
\SetArgSty{textrm}

\SetFuncSty{MyFuncSty}
\SetAlCapNameSty{MyFuncSty}

\SetKwSty{MyKwSty}

\SetCommentSty{MyCommentSty}
\SetKwInOut{Reference}{reference~}

\makeatletter
\newcount\my@repeat@count
\newcommand{\myrepeat}[2]{%
  \begingroup
  \my@repeat@count=\z@
  \@whilenum\my@repeat@count<#1\do{#2\advance\my@repeat@count\@ne}%
  \endgroup
}
\makeatother
\begin{document}

\author{Giulia Iezzi}
\address{Algebra, Geometry and Computer Algebra Group, RPTU University Kaiserslautern-Landau, Gottlieb-Daimler-Straße, Building 48, 67663 Kaiserslautern, Germany}
\email{giulia.iezzi@math.rptu.de}

\begin{abstract}

We define linear degenerations of Schubert varieties via a special class of quiver Grassmannians. To do so, we restrict our study to an appropriate subvariety in the variety of representations of the considered quiver and describe a base change action on this subvariety. We provide two explicit parametrisations for the orbits of this action, one of which encodes the partial order relations on such orbits.

\end{abstract}
{\let\newpage\relax\maketitle}

\section*{Acknowledgments}

The author would like to thank Ghislain Fourier and Martina Lanini for the helpful discussions and their guidance during this project.
Further thanks to Xin Fang, Evgeny Feigin and Markus Reineke for insightful discussions.
This work is a contribution to the SFB-TRR 195 “Symbolic Tools in Mathematics and their Applications” of the German Research Foundation (DFG) and is based on Chapters 2 and 5 of the author's PhD thesis \cite{mythesis}.

\section{Introduction}

Given a quiver $Q$ and a $Q$-representation $M$, the quiver Grassmannian $\Gr_{\bf e}(M)$ is the projective variety parametrising subrepresentations $N\subseteq M$ of dimension vector $\mathbf{e}$. Quiver Grassmannians first appeared in \cite{crawle-boevey1989quiver, schofield1992quiver} and have since been extensively studied, for instance as a tool in cluster algebra theory \cite{caldero2006cluster} or for studying linear degenerations of the flag variety \cite{fourier2020lineardegenerations,cerulliirelli2012quiveranddegenerate,feigin2010grassmanndegenerations, feigin2013frobeniussplittin}.
A key advantage of using quivers for studying problems of geometrical nature is the possibility to exploit combinatorial and algebraic tools to deduce geometrical properties of the considered projective variety.
A few examples of this can be found in results by Cerulli Irelli, Esposito, Franzen, Feigin and Reineke, for instance in \cite{cerulli2017schubert,irelli2021cell,irelli2013desingularization}.
However, as proven by Reineke in \cite{reineke2013projectiveisquiver} and, more generally, by Ringel in \cite{ringel2018quiver}, every projective variety arises as the quiver Grassmannian of any wild acyclic quiver.
%This implies that the study of quiver Grassmannians needs to be restricted in order to be meaningful, for instance by considering particular quivers or quiver representations.

The investigation of linear degenerations of flag varieties using quiver Grassmannians took place in this framework in the last fifteen years, appearing in several papers by Feigin, Finkelberg, Cerulli Irelli, Reineke, Fang, Fourier \cite{feigin2010grassmanndegenerations, cerulliirelli2012quiveranddegenerate,feigin2013frobeniussplittin,cerulli2017linear,fourier2020lineardegenerations}.
The word ``degeneration'' refers to a construction that allows us to regard a variety as a specific element in a family of varieties, or as a chosen fibre of a certain morphism; the above-cited works characterise several geometric and combinatorial aspects of linear degenerations of flag varieties, such as their defining equations, cellular decompositions and - making use of rank tuples - flatness, irreducibility and normality.

Schubert varieties first appeared at the end of the 19$^{\text{th}}$ century in the context of Schubert calculus, whose purpose is to determine the number of solutions of certain intersection problems, and have become some of the best understood examples of complex projective varieties.
They have recently been linked to degenerate flag varieties and quiver Grassmannians, for example in \cite{irelli2014degenerate}, where the authors show that any type A or C degenerate flag variety in the sense of \cite{feigin2010grassmanndegenerations} is isomorphic to a Schubert variety in an appropriate partial flag manifold, or later in \cite{cerulli2017schubert}, which proves that some Schubert varieties arise as irreducible components of certain quiver Grassmannians.
In \cite{iezzi2025quiver}, we considered a class of quiver Grassmannians which, depending on specific choices of the dimension vector, provide a realisation of smooth Schubert varieties or of Bott-Samelson resolutions of Schubert varieties in type A.

In this paper, we exploit the results obtained in \cite{iezzi2025quiver} to define linear degenerations of Schubert varieties via quiver Grassmannians.
To do so, we consider a specific quiver with relations $(\Gamma,I)$ and study the subvariety $R^{\iota}_{\bf d}$ in the variety of all representations of $(\Gamma,I)$.
Each representation in $R^{\iota}_{\bf d}$ is determined by a tuple of linear maps in $\prod_{j=1}^{n-1} U_{n+1}$, where $U_{n+1}$ is the subset of $\Mat_{n+1}$ (the set of square matrices of size $n+1$) consisting of upper-triangular matrices.
One notable property of this subvariety is the uniformity of the decompositions of any of its representations. We define a class of indecomposables of $(\Gamma,I)$, denoted by $U^{(h_1,\dots,h_n)}$, and prove the following:

\begin{maintheorem}[Theorem \ref{thm:indecompU}]\label{thm:1}
    All representations in $R^{\iota}_{\bf d}$ can be decomposed as direct sums of the indecomposable $(\Gamma,I)$-representations $U^{(h_1,\dots,h_n)}$.
\end{maintheorem}

Then, we answer the question of the representation type of $(\Gamma, I)$ - that is, whether there are finitely many isomorphism classes of indecomposable representations of $(\Gamma,I)$- when we restrict to the representations in $R^{\iota}_{\bf d}$ and to certain isomorphisms.
In order to do so, we consider the orbits in $R^{\iota}_{\bf d}$ under the action defined as
\begin{equation*}
    h \op M^f= (h_2 f_{n+1}^1 h_1^{-1}, h_3 f_{n+1}^2 h_2^{-1}, \dots, h_n f_{n+1}^{n-1} h_{n-1}^{-1}),
\end{equation*}
for some $h \in \prod_{j=1}^{n-1} B_{n+1}$, where $B_{n+1}$ is the Borel subgroup of invertible upper-triangular matrices inside the general linear group, and $M^f\in R^{\iota}_{\bf d}$.
To simplify notation, we refer to this action as $B$-action and to such orbits as $B$-orbits or, equivalently, as $B$-isomorphism classes in $R^{\iota}_{\bf d}$.
Notice, however, that the group acting on $R^{\iota}_{\bf d}$ is the product $\prod_{j=1}^{n-1} B_{n+1}$, acting via base change as given above.
In \cite{abeasis1985degenerations}, Abeasis and Del Fra parametrise the isomorphism classes of the representations of any quiver of type $\mathbb{A}_n$.
We employ an analogous parametrisation, denoted by ${\bf r}$, and make use of our previous results to prove the following:

\begin{maintheorem}[Theorem \ref{thm:firstparam}]\label{thm:2}
    Two representations $M^f, M^{g}$ in $R^{\iota}_{\bf d}$ are in the same $B$-isomorphism class if and only if ${\bf r}^f={\bf r}^{g}$.
\end{maintheorem}

We then provide an alternative parametrisation for such $B$-isomorphism classes. We notice that our $B$-action can be regarded as an expansion of the one considered in \cite{miller2005matrix} by Miller and Sturmfels in the context of matrix Schubert varieties.
There, the orbits are described in terms of ranks of certain submatrices, namely the north-west ranks. We adapt this parametrisation to our $B$-action by defining the south-west arrays, denoted by ${\bf s}$, and showing the following:

\begin{maintheorem}[Theorem \ref{thm:secondparam}]\label{thm:3}
    Two representations $M^f, M^{g}$ in $R^{\iota}_{\bf d}$ are in the same $B$-isomorphism class if and only if ${\bf s}^f={\bf s}^{g}$.
\end{maintheorem}

Finally, we recall the construction of linear degenerations of flag varieties following \cite{cerulli2017linear}: given an equioriented quiver of type $\mathbb{A}_n$, the linear degeneration of $\Flag_{n+1}$, denoted by $\Flag^{f}_{n+1}$, is defined as the quiver Grassmannian consisting of all subrepresentations of the following representation:
\begin{equation*}
\begin{tikzcd}[]
\overset{\C^{n+1}}{\bullet} \ar[r, "{f_1}"] & \overset{\C^{n+1}}{\bullet} \ar[r, "{f_2}"] & ... \ar[r, "{f_{n-1}}"] & \overset{\C^{n+1}}{\bullet}
\end{tikzcd}
\end{equation*}
for the dimension vector $(1,2,\dots,n)$. The quiver Grassmannian corresponding to the choice of $f$ can be regarded as the fibre of the projection $\pi: Y\to R$ over $f$, where $R$ is the variety $\Mat_{n+1}^{n-1}$ and $Y$ is the universal quiver Grassmannian, i.e. the variety of compatible pairs of sequences of maps and sequences of subspaces.
In particular, we recall a result about the flat locus of $\pi$, namely that it is the union of all orbits that degenerate to the orbit parametrised by a certain rank tuple.

We build upon the constructions and results obtained in Section \ref{sec:parametrisations} and in \cite{iezzi2025quiver} and define the linear degeneration of the Schubert variety $X_w$, denoted by $X^f_w$, as the fibre of $\pi: Y\to R^{\iota}_{\bf d}$ over $f$, where $R^{\iota}_{\bf d}$ is the subvariety of $(\Gamma,I)$-representations considered in Section \ref{sec:subvarietyR} and $Y$ is the universal quiver Grassmannian.
First, we show the following:
\begin{maintheorem}[Corollary \ref{cor:orbitdegen}]\label{thm:4}
    The orbit $\orb^{\iota}_{M^{g}}$ of $M^{g}$ under the $B$-action of $G^{\iota}_{\bf d}$ lies in the closure of the orbit $\orb^{\iota}_{M^f}$ of $M^f$ (with respect to the Zariski topology on $\Mat_{n+1}$) if and only if ${\bf s}^f\leq {\bf s}^{g}$, where the "less than or equal to" relation is intended componentwise on the south-west arrays. In this case, we write $\orb^{\iota}_{M^{g}} \subseteq \overline{\orb}^{\iota}_{M^f}$ and we say that $\orb^{\iota}_{M^{f}}$ degenerates to $\orb^{\iota}_{M^{g}}$. 
\end{maintheorem}

To conclude the discussion on the south-west parametrisation, we list the conditions that are necessary and sufficient for a tuple of non-negative integers to be the south-west array of some representation in $ R^{\iota}_{\bf d}$.

Lastly, we consider the natural question of determining the flat locus of the projection $\pi: Y\to R^{\iota}_{\bf d}$. We present a potential strategy and, together with some motivations, our conjecture:

\begin{mainconjecture}[Conjecture \ref{conj:flatlocus}]\label{conj}
    A tuple $f\in R^{\iota}_{\bf d}$ is in the flat locus of $\pi : Y\to R^{\iota}_{\bf d}$ if and only if $\dim(\Gr_{{\bf e}^w}(M^f))=\dim(X_w)$.
\end{mainconjecture}

The paper is organised as follows: in Section \ref{sec:backgroundquivers}, we recall basic facts about quiver representations and quiver Grassmannians. Section \ref{sec:qGandschubvar} is dedicated to Schubert varieties and to recalling their realisation and desingularisation via quiver Grassmannians obtained in \cite{iezzi2025quiver}. Section \ref{sec:subvarietyR} contains the main results of this paper: we consider the variety of representations $R^{\iota}_{\bf d}$ and prove Theorems \ref{thm:1}, \ref{thm:2} and \ref{thm:3}. In Section \ref{sec:lindeg}, we define linear degenerations of Schubert varieties and prove Theorem \ref{thm:4}. Lastly, in Section \ref{sec:flatlocus}, we formulate Conjecture \ref{conj} and present its motivation.

\section{Background on quiver representations}\label{sec:backgroundquivers}

We first collect a few facts about quiver representations and quiver Grassmannians. Some standard references are \cite{crawley1992lectures,sc2014quiver}.

\begin{definition}
    A finite \Def{quiver} $Q=(Q_0,Q_1,s,t)$ is given by a finite set of vertices $Q_0$, a finite set of arrows $Q_1$ and two maps $s,t: Q_1\to Q_0$ assigning to each arrow its source, resp. target.
 \end{definition}

 \begin{definition}
    A \Def{relation} on a quiver $Q$ is a subspace of the path algebra of $Q$ spanned by linear combinations of paths with common source and target, of length at least 2.
    Given a two-sided ideal $I$ of $\K Q$ generated by relations, the pair $(Q,I)$ is a \Def{quiver with relations} and the quotient algebra $\K Q/I$ is the path algebra of $(Q,I)$.
\end{definition}

A system of relations for $I$ is defined as a subset $R$ of $\cup_{i,j\in Q_0} iIj$, where $i$ denotes the trivial path on vertex $i$, such that $R$, but no proper subset of $R$, generates $I$ as a two-sided ideal. For any two vertices $i$ and $j$, we denote by $r(i,j,R)$ the cardinality of the set $R\cap iIj$, which contains those elements in $R$ that are linear combinations of paths starting in $i$ and ending in $j$. If $Q$ contains no oriented cycle, then the numbers $r(i,j,R)$ are independent of the chosen system of relations (see for instance \cite{bongartz1983algebras}), and can therefore be denoted by $r(i,j)$.

 \begin{definition}
     Given a quiver $Q$, the finite-dimensional $Q$-\Def{representation} $M$ over an algebraically closed field $\K$ is the ordered pair $((M_i)_{i\in Q_0},(M^{\alpha})_{\alpha\in Q_1})$, where $M_i$ is a finite-dimensional $\K$-vector space attached to vertex $i\in Q_0$ and $M^{\alpha}:M_{s(\alpha)}\to M_{t(\alpha)}$ is a $\K$-linear map for any $\alpha \in Q_1$.
     The \Def{dimension vector} of $M$ is $\dimvec M \coloneqq(\dim_{\K}M_i)_{i\in Q_0} \in \Z^{|Q_0|}_{\geq 0}$.
     A \Def{subrepresentation} of $M$, denoted by $N=((N_i)_{i\in Q_0},(M^{\alpha}\restriction_{N_{s(\alpha)}})_{\alpha\in Q_1})$, is a $Q$-representation such that $N_i\subseteq M_i $ for any $ i\in Q_0$ and $M^{\alpha}(N_{s(\alpha)})\subseteq N_{t(\alpha)}$ for any $\alpha\in Q_1$.
     \end{definition}

From now on, we consider only bound quivers in the sense of Schiffler (see \cite[Definition 5.1]{sc2014quiver}). 
The finite-dimensional $Q$-representations over $\K$ form a category, denoted by $\rep_{\K}(Q)$, where a morphism $\phi$ between $M$ and $M'$ in $\rep_{\K}(Q)$ is given by linear maps $\phi_i:M_i\to M'_i \quad \forall i \in Q_0$ such that $\phi_{t(\alpha)} \circ M^{\alpha}=M'^{\alpha}\circ \phi_{s(\alpha)}$.
Similarly, the category $\rep_{\K}(Q,I)$ consists of the finite-dimensional representations of $Q$ that satisfy the relations in $I$.
It is known (see \cite[Theorem 5.4]{sc2014quiver} for a proof) that $\rep_{\K}(Q)$ is equivalent to the category $A$-$\modCat$ of finite-dimensional modules over the path algebra $A=\K Q$ of $Q$. 
Furthermore, $\rep_{\K}(Q)$ is Krull-Schmidt (\cite[Theorem 1.2]{sc2014quiver}) and hereditary (\cite[Theorem 2.24]{sc2014quiver}).

\begin{definition}\label{def:projinj}
    A representation $P\in \rep_{\K}(Q)$ is called \Def{projective} if the functor $\Hom(P,-)$ maps surjective morphisms to surjective morphisms. Dually, $I\in \rep_{\K}(Q)$ is called \Def{injective} if the functor $\Hom(-,I)$ maps injective morphisms to injective morphisms.
\end{definition}

If $Q$ is a quiver without oriented cycles, then to each vertex $i\in Q_0$ corresponds exactly one indecomposable projective representation, denoted by $P(i)$. Such projective representations are easy to describe: the basis of the vector space $P(i)_k$ at vertex $k$ is given by the set of all possible paths from vertex $i$ to vertex $k$, and the actions of the maps between the vector spaces are induced by the concatenation of paths.
Dually, for every vertex $i\in Q_0$ there is exactly one indecomposable injective representation $I(i)$, whose basis for each vector space $I(i)_k$ is given by the set of all possible paths from vertex $k$ to vertex $i$ and whose maps act by concatenation of paths.

The following result holds in any additive category.
\begin{proposition}\cite[Proposition 2.7]{sc2014quiver}\label{prop:sumsofprojinj}
    Let $P,P',I$ and $I'$ be representations of $Q$. Then:
    \begin{enumerate}
        \item $P\oplus P'$ is projective $\iff$ $P$ and $P'$ are projective;
        \item $I\oplus I'$ is injective $\iff$ $I$ and $I'$ are injective.
    \end{enumerate}
\end{proposition}

More background and details can be found for instance in \cite[Section 2.2]{bongartz1983algebras} or \cite{derksen2002semi}. To simplify notation, we will sometimes denote a $Q$-representation $M$ by its tuple of vector spaces $(M_i)_{i \in Q_0}$ when the assignment of the linear maps is clear from context.

\begin{definition}
    Consider a quiver $Q$, a $Q$-representation $M$ and a dimension vector ${\bf e}\in\Z^{Q_0}_{\geq 0}$ such that $e_i\leq \dim M_i \; \forall i\in Q_0$. The \Def{quiver Grassmannian} $\Gr_{\bf e}(M)$ parametrises the subrepresentations $N$ of $M$ with $\dim N_i=e_i $ for all $ i\in Q_0$.
\end{definition}

    Analogous to Grassmannians and flag varieties, non-empty quiver Grassmannians can be realised as closed subvarieties of products of Grassmannians, via the closed embedding
    \begin{equation*}
        \iota: \Gr_{\bf e}(M) \to \prod_{i \in Q_0}\Gr(e_i,M_i)
    \end{equation*}
    which sends a subrepresentation $N$ of $M$ to the collection of $e_i$-dimensional subspaces $N_i$ of $M_i$.
    %The relations defining (pointwise) the subvariety associated to a given quiver representation and a dimension vector are given in \cite{lorscheid2019pluckerelations}.

\begin{example}[The flag variety of type A]\label{ex: flag variety}

    In \cite[Proposition 2.7]{cerulliirelli2012quiveranddegenerate}, the authors realise the (linear degenerate) flag variety of type A as quiver Grassmannians associated to certain representations of the equioriented quiver of type $\mathbb{A}_n$. In particular, the complete flag variety $\Flag_{n+1}$ in $\C^{n+1}$ can be realised as follows.

    Consider the quiver with $n$ vertices, labelled from 1 to $n$, and $n-1$ arrows of the form $i\to i+1$. We fix the  dimension vector ${\bf e}=(1,2,\dots,n)$ and the representation $M$ with $M_i=\C^{n+1}$ for $i=1,\dots,n$ and $M^{\alpha}=\id$ for all arrows $\alpha$:
\begin{equation*}
\begin{tikzcd}[]
\overset{\C^{n+1}}{\bullet} \ar[r, "\id"] & \overset{\C^{n+1}}{\bullet} \ar[r, "\id"] & ... \ar[r, "\id"] & \overset{\C^{n+1}}{\bullet}.\\
\end{tikzcd}
\vspace{-0.7cm}
\end{equation*}
 The quiver Grassmannian $\Gr_{\bf e}(M)$ consists precisely of the subrepresentations $N$ of $M$ with $\dim(N_i)=i$, i.e. full flags of vector subspaces.
\end{example}

\subsection{The variety of $Q$-representations}

In Section \ref{sec:backgroundquivers}, we defined a quiver representation $M \in \rep_{\K}(Q)$ as the assignment of a finite-dimensional $\K$-vector space $M_i$ to each vertex in $Q_0$ and of a linear map $M_{\alpha}$ to each arrow in $Q_1$. By choosing bases, we can identify the vector spaces $M_i$ with $\K^{d_i}$, where $d_i=\dim M_i$, and consequently we can represent each linear map $M_{\alpha}$, for $\alpha: i \to j$, with a $d_j \times d_i$ matrix. This means that, given a quiver $Q$ and a dimension vector $\bf d$, determining a $Q$-representation of dimension vector $\bf d$ is equivalent to giving the linear maps $M_{\alpha}$.

\begin{definition}\label{def:repspace}
    The \Def{representation space} of the quiver $Q$ for the dimension vector $\bf d$ is
    \begin{equation*}
        R_{\bf d} \cong \bigoplus_{\alpha : i\to j} \Hom(\K^{d_i}, \K^{d_j}) \cong \bigoplus_{\alpha : i\to j} \Mat_{d_j \times d_i}(\K).
    \end{equation*}
\end{definition}
The representation space $R_{\bf d}$ is isomorphic by definition to the affine space $\A^r$, with $r\coloneqq \sum_{\alpha: i\to j} d_i \times d_j$. We consider then the group
\begin{equation*}
    G_{\bf d}=\prod_{i\in Q_0} \GL_{d_i}(\K),
\end{equation*}

where $\GL_{d_i}(\K)$ denotes the group of invertible $d_i\times d_i$ matrices, and define the action of $G_{\bf d}$ on $R_{\bf d}$ as
\begin{equation}\label{eq:Gaction}
    g \op M= g_{t(\alpha)} M_{\alpha} g^{-1}_{s(\alpha)}
\end{equation}
for some $g \coloneqq (g_i)_{i\in Q_0} \in G_{\bf d}$ and $M\in R_{\bf d}$. We write $\orb_M$ for the orbit of a representation $M$ under this action, that is $\orb_M=\{g \op M | g\in G_{\bf d} \}$.

\begin{example}\cite[Example 8.1]{schiffler2014quiver}
    Consider the quiver $Q=\begin{tikzcd}[]
\overset{1}{\bullet} \ar[r, "\alpha"] & \overset{2}{\bullet}
\end{tikzcd}$ and a dimension vector ${\bf d}=(d_1, d_2)$. In this case, the representation space $R_{\bf d}$ is isomorphic to $\Mat_{d_2 \times d_1}(\K)$, the elements $g=(g_1,g_2)$ of $G_{\bf d}$ are pairs of invertible matrices of size $d_1$ and $d_2$ respectively, and the orbits of the action defined in \ref{eq:Gaction} are $\orb_M=\{g_2 M_{\alpha} g_1^{-1} | (g_1,g_2)\in G_{\bf d} \}$. In other words, $\orb_M$ is the set of all matrices whose rank is equal to the rank of $M_\alpha$. 
\end{example}

\begin{lemma}\cite[Lemma 8.1]{schiffler2014quiver}\label{lemma:isomGL}
    The orbit $\orb_M$ is precisely the isomorphism class of the representation $M$, that is
    \begin{equation*}
        \orb_M=\{M'\in \rep_{\K}(Q) | M'\cong M \}.
    \end{equation*}
\end{lemma}

Similarly, the stabiliser $\Stab(M)=\{g\in G_{\bf d} \; | \; g\op M = M \}$ of a $Q$-representation $M$ is isomorphic to the automorphism group $\Aut(M)$ of $M$.

\begin{lemma}\cite[Lemma 8.2]{schiffler2014quiver}
    Let ${\bf d}\in \Z^n$. Then
    \begin{enumerate}
        \item for any representation $M$ of dimension vector ${\bf d}$, the dimensions of the varieties $\orb_M,G_{\bf d}$ and $\Aut(M)$ satisfy
        \begin{equation*}
            \dim \orb_M = \dim G_{\bf d} - \dim \Aut(M);
        \end{equation*}
        \item there is one orbit of codimension zero in $R_{\bf d}$.
    \end{enumerate}
\end{lemma}

\section{Schubert varieties as quiver Grassmannians}\label{sec:qGandschubvar}

Given any $v_1,\dots,v_r$ in $\C^{n+1}$, we denote by $\langle v_1,\dots, v_r \rangle$ their $\C$-linear span. To define Schubert varieties in $\Flag_{n+1}$, we first fix a basis $\mathcal{B}=\set{b_1, b_2,\dots, b_{n+1}}$ of $\C^{n+1}$ and denote by $F_{\bullet}$ the standard flag $\langle b_1\rangle \subseteq  \langle b_1, b_2\rangle \subseteq \dots \subseteq \langle b_1, b_2,\dots, b_{n+1} \rangle$ and by $S_{n+1}$ the symmetric group on $n+1$ elements. More facts and details about Schubert varieties can be found for instance in \cite[Part III]{fulton1997young}.

\begin{definition}\label{def:schubcell}
For $w \in S_{n+1}$, the \Def{Schubert cell} $X_w^\circ$ is
\begin{equation*}
X^\circ_w = \set{V_{\bullet}\in \Flag_{n+1}: \dim(F_p\cap V_q)= \#\set{k\leq q : w(k)\leq p},\; 1 \leq p, q \leq n+1}.
\end{equation*}
\end{definition}

\begin{definition}\label{def:schubvar}
The \Def{Schubert variety} $X_w$ is defined as the closure in $\Flag_{n+1}$ of the cell $X_w^\circ$, that is
\begin{equation*}
X_w = \set{V_{\bullet}\in \Flag_{n+1}: \dim(F_p\cap V_q)\geq \#\set{k\leq q : w(k)\leq p}, \;1 \leq p, q \leq n+1}.
\end{equation*}
\end{definition}

We represent a permutation $w$ in $S_n$ by listing its (naturally) ordered images, that is, its one-line notation $w=[w(1)w(2)\dots w(n)]$.

%This is because we want to build a correspondance between the $d_{p,q}$ and the entries of a matrix, where the row index $p$ corresponds to the subspace $V_p$ and $q$ to $F_q$.

\begin{example}
For $e=[1\,2\dots n+1]$ and $w_0= [n+1 \, n\dots 1]$ in $S_{n+1}$, it is easy to compute from Definition \ref{def:schubvar} the Schubert varieties of minimal and maximal dimension, respectively $X_e = \{ F_{\bullet} \}$ and $X_{w_0}=\Flag_{n+1}$.
\end{example}

%\subsection{Schubert varieties and smoothness}

A combinatorial characterisation of smooth Schubert varieties was provided in \cite{lakshmibai1990criterion}: a Schubert variety $X_w$ is smooth if and only if $w$ avoids the patterns $[4231]$ and $[3412]$. We recall that a permutation $w=[w(1)w(2)\dots w(n)]$ avoids a pattern $\pi$ if no subsequence of $w$ has the same relative order as the entries of $\pi$.
In \cite[Theorem 1.1]{gasharov2002cohomology}, the authors prove that this pattern-avoiding condition is equivalent to $X_w$ being defined by non-crossing inclusions:

\begin{definition}(\cite[Section 1]{gasharov2002cohomology})\label{def:inclusions}
A Schubert variety $X_w$ is \Def{defined by inclusions} if the defining conditions on each $V_q$ (see Definition \ref{def:schubvar}) are a conjunction of conditions of the form $V_q \subseteq F_p$ and $V_q \supseteq F_{s}$, for some $p$ and $s$. A pair of conditions $V_{q}\subset F_{p}$ and $F_{p'}\subset V_{q'} $ is \Def{crossing} if $q<q'$ and $p>p'$.
If $X_w$ is defined by inclusions and its conditions do not contain any crossing pair, then $X_w$ is \Def{defined by non-crossing inclusions}.
\end{definition}

\begin{comment}
\begin{example}
All permutations in $S_3$ are defined by non-crossing inclusions.

In $S_5$, the permutation $w=[31542]$ avoids both patterns $[4231]$ and $[3412]$, which means that $X_{w}$ is defined by non-crossing inclusions. We can compute these inclusions using Definition \ref{def:schubvar}: a flag $V_{\bullet}$ is in $X_{w}$ if and only if
\begin{equation*}
    V_1 \subseteq F_3,\; F_1 \subseteq V_2\subseteq F_3,\; F_1\subseteq V_3,\; F_1\subseteq V_4 .
\end{equation*}
The same conditions can be described without redundancy as $F_1 \subseteq V_2\subseteq F_3$, which is a pair of non-crossing inclusions.

A permutation in $S_5$ that yields crossing inclusions is \mbox{$\tau=[45312]$}, which contains the pattern $[3412]$. A flag $V_{\bullet}$ is in $X_{\tau}$ if and only if $V_1\subseteq F_4$ and $ F_1\subseteq V_4$.
Finally, the permutation $\pi=[53421]$ in $S_5$ contains the pattern $[4231]$ and defines a non-trivial condition on $X_{\pi}$ that is not an inclusion: a flag $V_{\bullet}$ is in $X_{\pi}$ if and only if $\dim(F_3\cap V_2)\geq 1$.
\end{example}
\end{comment}

In \cite[Section 4]{iezzi2025quiver}, we defined the quiver with relations $(\Gamma, I)$:

\begin{equation}\label{eq:GammaI}
\begin{tikzcd}[]
\overset{(1,1)}{\bullet} \ar[r, ""] \ar[d, ""]  & \overset{(1,2)}{\bullet} \ar[r, ""] \ar[d, ""] & ...\ar[r, ""] \ar[d, ""] & \overset{(1,n)}{\bullet} \ar[d, ""] \\
\overset{(2,1)}{\bullet} \ar[ur, phantom, "\scalebox{1.5}{$\circlearrowleft$}"] \ar[r, ""] \ar[d, ""] & \overset{(2,2)}{\bullet} \ar[ur, phantom, "\scalebox{1.5}{$\circlearrowleft$}"] \ar[r, ""] \ar[d, ""] & ... \ar[ur, phantom, "\scalebox{1.5}{$\circlearrowleft$}"] \ar[r, ""] \ar[d, ""] & \overset{(2,n)}{\bullet} \ar[d, ""] \\
... \ar[ur, phantom, xshift=5, "\scalebox{1.5}{$\circlearrowleft$}"] \ar[r, ""] \ar[d, ""] & ... \ar[ur, phantom,xshift=3, "\scalebox{1.5}{$\circlearrowleft$}"] \ar[r, ""] \ar[d, ""] & ... \ar[ur, phantom,xshift=2, "\scalebox{1.5}{$\circlearrowleft$}"] \ar[r, ""] \ar[d, ""] & ... \ar[d, ""]\\
\underset{(n+1,1)}{\bullet} \ar[ur, phantom, "\scalebox{1.5}{$\circlearrowleft$}"] \ar[r, ""] & \underset{(n+1,2)}{\bullet} \ar[ur, phantom, "\scalebox{1.5}{$\circlearrowleft$}"] \ar[r, ""] & ... \ar[ur, phantom, "\scalebox{1.5}{$\circlearrowleft$}"] \ar[r, ""] & \underset{(n+1,n)}{\bullet}
\end{tikzcd}
\end{equation}

and its representation $M=((M_{i,j})_{(i,j)\in \Gamma_0},(M_{\alpha})_{\alpha \in \Gamma_1})$, given by 
\begin{equation*}
M_{i,j}=\C^i, \quad
M_{\alpha}=\begin{cases}
    \iota_{i+1,i} & \text{ if } s(\alpha)=(i,j), t(\alpha)=(i+1,j) \\
    \id & \text{ if } s(\alpha)=(i,j), t(\alpha)=(i,j+1)
\end{cases}
\end{equation*}
where $\iota_{i+1,i}$ denotes the inclusion of $\C^i$ into $\C^{i+1}$, represented with respect to the chosen basis $\mathcal{B}=\set{b_1, b_2,\dots, b_{n+1}}$ by the matrix
\begin{equation}\label{eq:iota}
    \iota_{i+1,i}=\begin{bsmallmatrix}
1&0&\dots&0\\
0&1&\dots&0\\
\dots&\dots&\dots&\dots\\
0&0&0&1\\
0&0&0&0
\end{bsmallmatrix}.
\end{equation}

%We write the dimension vector of $M$ as ${\bf d} =(d_{i,j})$, with $d_{i,j}= i$, for $i=1,\dots,n+1$ and $j=1,\dots,n$.

The relations imposed on $\Gamma$ are trivially satisfied by the representation $M$:
\begin{equation*}
\begin{tikzcd}[sep=large]
\overset{\C}{\bullet} \ar[r, "\id"] \ar[d, "\iota_{2,1}"]  & \overset{\C}{\bullet} \ar[r, "\id"] \ar[d, "\iota_{2,1}"] & ...\ar[r, "\id"] \ar[d, "\iota_{2,1}"] & \overset{\C}{\bullet} \ar[d, "\iota_{2,1}"] \\
\overset{\C^2}{\bullet} \ar[ur, phantom, "\scalebox{1.5}{$\circlearrowleft$}"] \ar[r, "\id"] \ar[d, "\iota_{3,2}"] & \overset{\C^2}{\bullet} \ar[ur, phantom, "\scalebox{1.5}{$\circlearrowleft$}"] \ar[r, "\id"] \ar[d, "\iota_{3,2}"]  & ... \ar[ur, phantom,xshift=5, "\scalebox{1.5}{$\circlearrowleft$}"] \ar[r, "\id"] \ar[d, "\iota_{3,2}"] & \overset{\C^2}{\bullet} \ar[d, "\iota_{3,2}"] \\
... \ar[ur, phantom,xshift=5, "\scalebox{1.5}{$\circlearrowleft$}"] \ar[r, "\id"] \ar[d, "\iota_{n+1,n}"] & ... \ar[ur, phantom,xshift=5, "\scalebox{1.5}{$\circlearrowleft$}"] \ar[r, "\id"] \ar[d, "\iota_{n+1,n}"] & ... \ar[ur, phantom, xshift=5, "\scalebox{1.5}{$\circlearrowleft$}"] \ar[r, "\id"] \ar[d,yshift=7, "\iota_{n+1,n}"] & ... \ar[d, "\iota_{n+1,n}"]\\
\overset{\C^{n+1}}{\bullet} \ar[ur, phantom, "\scalebox{1.5}{$\circlearrowleft$}"] \ar[r, "\id"] & \overset{\C^{n+1}}{\bullet} \ar[ur, phantom, "\scalebox{1.5}{$\circlearrowleft$}"] \ar[r, "\id"] & ... \ar[ur, phantom, xshift=5,yshift=5, "\scalebox{1.5}{$\circlearrowleft$}"] \ar[r, "\id"] & \overset{\C^{n+1}}{\bullet}\\
\end{tikzcd}.
\end{equation*}

\vspace{-1cm}

After showing that every quiver Grassmannian associated to the $(\Gamma,I)$-representation $M$ is a smooth and irreducible projective variety, we defined two different, appropriate dimension vectors ${\bf r}^w$ and ${\bf e}^w$ for $(\Gamma,I)$ and proved the following results:

\begin{theorem}[Theorem 5.19, \cite{iezzi2025quiver}]
  The quiver Grassmannian $\Gr_{{\bf r}^w}(M)$ is isomorphic to any Bott-Samelson resolution of $X_w$ associated to a geometrically compatible decomposition of $w$.
\end{theorem}

\begin{theorem}[Theorem 6.4, \cite{iezzi2025quiver}]
    If $w\in S_{n+1}$ avoids the patterns $[4231]$ and $[3412]$, the quiver Grassmannian $\Gr_{{\bf e}^w}(M)$ is isomorphic to the Schubert variety $X_w$.
\end{theorem}

In Section \ref{sec:lindegschubvar} of this paper, we will exploit these results for defining linear degenerations of Schubert varieties. In order to do this, we study in the following section the subvariety of representations of $(\Gamma,I)$ that correspond to such degenerations.

\section{The variety $R^{\iota}_{\bf d}$ of $(\Gamma,I)$-representations}\label{sec:subvarietyR}

In this paper, we restrict the question about the representation type of $(\Gamma,I)$ to certain representations in $\rep_{\C}(\Gamma,I)$ and show that they can be decomposed as direct sums of indecomposables belonging to a finite class.
The motivation for such a restriction will be explained in Section \ref{sec:lindegschubvar}, where we introduce our definition of linear degenerations of Schubert varieties using quiver Grassmannians.

\begin{definition}\label{def:subvariety}
    Let $(\Gamma, I)$ be the quiver defined in \eqref{eq:GammaI} and $R_{\bf d}$ the associated variety of representations (see Definition \ref{def:repspace}). We denote by $R^{\iota}_{\bf d}$ the subvariety of $R_{\bf d}$ defined as
    \begin{equation*}
        R^{\iota}_{\bf d}=\{M \in  R_{\bf d} | M_{\alpha}=\iota_{i+1,i} \; \forall \alpha: (i,j)\to (i+1,j), i=1,\dots,n, j=1,\dots,n\}.
    \end{equation*}
\end{definition}

    In words, we obtain $R^{\iota}_{\bf d}$ by setting all the linear maps associated to the vertical arrows of $(\Gamma,I)$ to the standard inclusion of $\C^i$ into $\C^{i+1}$ (see the definition in \eqref{eq:iota}).

\begin{remark}\label{remark:subvariety}
    Because of the commutativity relations on $(\Gamma,I)$, the linear maps of a representation $M$ in $R^{\iota}_{\bf d}$ associated to the horizontal arrows of $(\Gamma, I)$ are not independent of one another. Let us denote by $f_{i}^j$ the matrix representation of the linear map $M_{\alpha}$, for $\alpha$ such that $s(\alpha)=(i,j)$ and $t(\alpha)=(i,j+1)$, with respect to the chosen basis $\mathcal{B}=\set{b_1, b_2,\dots, b_{n+1}}$. The representation $M$ is then
\begin{equation*}
\begin{tikzcd}[sep=large]
\overset{\C}{\bullet} \ar[r, "\textcolor{forestgreen}{f_{1}^1}"] \ar[d, "\iota_{2,1}"]  & \overset{\C}{\bullet} \ar[r, "\textcolor{forestgreen}{f_{1}^2}"] \ar[d, "\iota_{2,1}"] & ...\ar[r, "\textcolor{forestgreen}{f_{1}^{n-1}}"] \ar[d, "\iota_{2,1}"] & \overset{\C}{\bullet} \ar[d, "\iota_{2,1}"] \\
\overset{\C^2}{\bullet} \ar[ur, phantom, "\scalebox{1.5}{$\circlearrowleft$}"] \ar[r, "\textcolor{forestgreen}{f_{2}^1}"] \ar[d, "\iota_{3,2}"] & \overset{\C^2}{\bullet} \ar[ur, phantom, "\scalebox{1.5}{$\circlearrowleft$}"] \ar[r, "\textcolor{forestgreen}{f_{2}^2}"] \ar[d, "\iota_{3,2}"]  & ... \ar[ur, phantom,xshift=5, "\scalebox{1.5}{$\circlearrowleft$}"] \ar[r, "\textcolor{forestgreen}{f_{2}^{n-1}}"] \ar[d, "\iota_{3,2}"] & \overset{\C^2}{\bullet} \ar[d, "\iota_{3,2}"] \\
... \ar[ur, phantom,xshift=5, "\scalebox{1.5}{$\circlearrowleft$}"] \ar[r, "\textcolor{forestgreen}{f_{i}^1}"] \ar[d, "\iota_{n+1,n}"] & ... \ar[ur, phantom,xshift=5, "\scalebox{1.5}{$\circlearrowleft$}"] \ar[r, "\textcolor{forestgreen}{f_{i}^j}"] \ar[d, "\iota_{n+1,n}"] & ... \ar[ur, phantom, xshift=5, "\scalebox{1.5}{$\circlearrowleft$}"] \ar[r, "\textcolor{forestgreen}{f_{i}^{n-1}}"] \ar[d,yshift=5, "\iota_{n+1,n}"] & ... \ar[d, "\iota_{n+1,n}"]\\
\overset{\C^{n+1}}{\bullet} \ar[ur, phantom, "\scalebox{1.5}{$\circlearrowleft$}"] \ar[r, "\textcolor{forestgreen}{f_{n+1}^1}"] & \overset{\C^{n+1}}{\bullet} \ar[ur, phantom, "\scalebox{1.5}{$\circlearrowleft$}"] \ar[r, "\textcolor{forestgreen}{f_{n+1}^2}"] & ... \ar[ur, phantom, xshift=5,yshift=5, "\scalebox{1.5}{$\circlearrowleft$}"] \ar[r, "\textcolor{forestgreen}{f_{n+1}^{n-1}}"] & \overset{\C^{n+1}}{\bullet}\\
\end{tikzcd}.
\end{equation*}
\vspace{-1cm}

We write $f_{1}^1=\begin{bsmallmatrix}
        a
    \end{bsmallmatrix}$ for some $a\in \C$, and applying $\iota_{2,1}$ after $f_{1}^1$ we get $\iota_{2,1}f_{1}^1=\begin{bsmallmatrix}
        a\\
        0
    \end{bsmallmatrix}$.
The relation $f_{2}^1\iota_{2,1}=\iota_{2,1}f_{1}^1$ implies $f_{2}^1=\begin{bsmallmatrix}
        a & b\\
        0 & c
    \end{bsmallmatrix}$, for some $b,c\in \C$. Similarly, we see that $f_{3}^1$ must be $f_{3}^1=\begin{bsmallmatrix}
        a & b & d\\
        0 & c & e\\
        0 & 0 & l
    \end{bsmallmatrix}$ for some $d,e,l\in \C$, and so forth for all $i=1,\dots,n+1$ and $j=1,\dots,n-1$. This means that a $(\Gamma,I)$-representation $M$ in the subvariety $R^{\iota}_{\bf d}$ is determined by the choice of $n-1$ upper-triangular matrices of size $n+1$. We denote this tuple by
    \begin{equation}\label{eq:ftuple}
        f=(f_{n+1}^1,\dots,f_{n+1}^{n-1})\in \prod_{j=1}^{n-1} U_{n+1},
    \end{equation}
    where $U_{n+1}$ is the subset of $\Mat_{n+1}$ consisting of upper-triangular matrices (different from the subgroup $B_{n+1}$ of invertible upper-triangular matrices). From now on, we will focus on studying the $(\Gamma,I)-$representations in the space $R^{\iota}_{\bf d}$, in particular by describing their decompositions into indecomposable representations of $(\Gamma, I)$ and how they can be parametrised. Since the representations in $R^{\iota}_{\bf d}$ are determined by the choice of $f$, we will denote them by $M^f$ and, depending on the context, identify $M^f$ with the tuple $(f_{n+1}^1,\dots,f_{n+1}^{n-1})$ and therefore write $M^f=(f_{n+1}^1,\dots,f_{n+1}^{n-1})$.
\end{remark}

Let us denote by $U^{(h_1,\dots,h_n)}$ the indecomposable representation of $(\Gamma,I)$ defined as:

\begin{equation}\label{eq:indecompUspaces}
    U^{(h_1,\dots,h_n)}_{i,j}\coloneqq\begin{cases}
        0 & \text{ if } i \leq n+1 - h_j \\
        \C & \text{ if } i > n+1 - h_j
    \end{cases},
\end{equation}
where $1\leq h_j\leq n+1$, for $i=1,\dots,n+1$ and $1\leq j \leq n$, and whose linear maps are

\begin{equation}\label{eq:indecompUmaps}
    U^{(h_1,\dots,h_n)}_{\alpha}\coloneqq \begin{cases}
        \id & \text{ if } U^{(h_1,\dots,h_n)}_{s(\alpha),t(\alpha)}=\C \\
        0 & \text{ otherwise }
    \end{cases}.
\end{equation}

Each index $h_j$ represents the ``height'' of the first nonzero vector space in the $j$-th column of $(\Gamma, I)$: this space and all the spaces below it are isomorphic to $\C$ by the definition given in \eqref{eq:indecompUspaces}. 
In order to satisfy the commutativity relations of $(\Gamma,I)$, we further require $ h_j\leq h_{j'}$ for any $j\leq j'$ and $h_{j'}> 0$. It is straightforward to verify that, if $h_{j'} >0$, the relation $h_j>h_{j'}$ implies the existence of the following diagram in $(\Gamma,I)$:
\begin{equation*}
\begin{tikzcd}[sep=large]
\overset{\C}{\bullet} \ar[r, "{[0]}"] \ar[d, "\id"]  & \overset{0}{\bullet} \ar[d, "{[0]}"] \\
\overset{\C}{\bullet} \ar[r, "\id"] & \overset{\C}{\bullet}
\end{tikzcd},
\end{equation*}
which does not commute.
As shown in Theorem \ref{thm:indecompU}, the definition for the linear maps of $U^{(h_1,\dots,h_n)}$ given in \eqref{eq:indecompUmaps} ensures that these representations of $(\Gamma, I)$ are actually indecomposable.

\begin{example}
Two examples of indecomposable $(\Gamma, I)$-representations of the form $U^{(h_1,\dots,h_n)}$ are $U^{(1,2)}$ and $U^{(2,2,4)}$, respectively for $n+1=3$ and $n+1=4$:
\begin{equation*}
U^{(1,2)}=
\begin{tikzcd}[sep=large]
\overset{0}{\bullet} \ar[r, "0"] \ar[d, "0"]  & \overset{0}{\bullet} \ar[d, "0"] \\
\overset{0}{\bullet} \ar[ur,phantom, "\scalebox{1.5}{$\circlearrowleft$}"] \ar[r, "0"] \ar[d, "0"] & \overset{\C}{\bullet}  \ar[d, "\id"] \\
\overset{\C}{\bullet} \ar[ur, phantom, "\scalebox{1.5}{$\circlearrowleft$}"] \ar[r, "\id"] & \overset{\C}{\bullet}
\end{tikzcd},\quad\quad
U^{(2,2,4)}=
\begin{tikzcd}[sep=large]
\overset{0}{\bullet} \ar[r, "0"] \ar[d, "0"]  & \overset{0}{\bullet} \ar[r, "0"] \ar[d, "0"] & \overset{\C}{\bullet} \ar[d, "\id"] \\
\overset{0}{\bullet} \ar[ur, phantom, "\scalebox{1.5}{$\circlearrowleft$}"] \ar[r, "0"] \ar[d, "0"] & \overset{0}{\bullet} \ar[ur, phantom, "\scalebox{1.5}{$\circlearrowleft$}"] \ar[r, "0"] \ar[d, "0"]  & \overset{\C}{\bullet} \ar[d, "\id"]\\
\overset{\C}{\bullet} \ar[ur, phantom, "\scalebox{1.5}{$\circlearrowleft$}"] \ar[r, "\id"] \ar[d, "\id"] & \overset{\C}{\bullet} \ar[ur, phantom, "\scalebox{1.5}{$\circlearrowleft$}"] \ar[r, "\id"] \ar[d, "\id"] & \overset{\C}{\bullet} \ar[d, "\id"]\\
\overset{\C}{\bullet} \ar[ur, phantom, "\scalebox{1.5}{$\circlearrowleft$}"] \ar[r, "\id"] & \overset{\C}{\bullet} \ar[ur, phantom, "\scalebox{1.5}{$\circlearrowleft$}"] \ar[r, "\id"] & \overset{\C}{\bullet}\\
\end{tikzcd}.
\end{equation*}
\end{example}

\begin{example}\label{ex:f_1f_2}
 For $n+1=4$, the representations $M^f$ in $R^{\iota}_{\bf d}$ are determined by $f=(f_4^1,f_4^2)\in U_4\times U_4$. To simplify our notation, when possible, we omit the row index and only specify the ``column'' of $(\Gamma,I)$ where a linear map appears. In this case, for example, we write $f=(f_1,f_2)\in U_4\times U_4$. A possible choice is:
    \begin{equation*}
        f_1= \begin{bsmallmatrix}
        0 & 0 & 0 & 0 \\
        0 & 1 & 0 & 0 \\
        0 & 0 & 0 & 0 \\
        0 & 0 & 0 & 1
    \end{bsmallmatrix},
        \quad f_2= \begin{bsmallmatrix}
        1 & 0 & 0 & 0 \\
        0 & 0 & 0 & 0 \\
        0 & 0 & 1 & 0 \\
        0 & 0 & 0 & 1
    \end{bsmallmatrix},
    \end{equation*}
    which corresponds to the representation
    \begin{equation*}
  M^f=
    \begin{tikzcd}[sep=large, ampersand replacement=\&]
\overset{\C}{\bullet} \ar[r, "{[0]}"] \ar[d, "\iota_{2,1}"]  \& \overset{\C}{\bullet} \ar[r, "{[1]}"] \ar[d, "\iota_{2,1}"] \& \overset{\C}{\bullet} \ar[d, "\iota_{2,1}"] \\
\overset{\C^2}{\bullet} \ar[ur, phantom, "\scalebox{1.5}{$\circlearrowleft$}"] \ar[r, "{\begin{bsmallmatrix}
        0 & 0 \\
        0 & 1 
    \end{bsmallmatrix}}"] \ar[d, "\iota_{3,2}"] \& \overset{\C^2}{\bullet} \ar[ur, phantom, "\scalebox{1.5}{$\circlearrowleft$}"] \ar[r, "{\begin{bsmallmatrix}
        1 & 0 \\
        0 & 0 
    \end{bsmallmatrix}}"] \ar[d, "\iota_{3,2}"]  \& \overset{\C^2}{\bullet} \ar[d, "\iota_{3,2}"]\\
\overset{\C^3}{\bullet} \ar[ur, phantom, yshift=5, "\scalebox{1.5}{$\circlearrowleft$}"] \ar[r, "{\begin{bsmallmatrix}
        0 & 0 & 0 \\
        0 & 1 & 0 \\
        0 & 0 & 0 
    \end{bsmallmatrix}}"] \ar[d, "\iota_{4,3}"] \& \overset{\C^3}{\bullet} \ar[ur, phantom, yshift=5,"\scalebox{1.5}{$\circlearrowleft$}"] \ar[r, "{\begin{bsmallmatrix}
        1 & 0 & 0 \\
        0 & 0 & 0 \\
        0 & 0 & 1
    \end{bsmallmatrix}}"] \ar[d, "\iota_{4,3}"] \& \overset{\C^3}{\bullet} \ar[d, "\iota_{4,3}"]\\
\overset{\C^4}{\bullet} \ar[ur, phantom,yshift=9, "\scalebox{1.5}{$\circlearrowleft$}"] \ar[r, "{\begin{bsmallmatrix}
        0 & 0 & 0 & 0 \\
        0 & 1 & 0 & 0 \\
        0 & 0 & 0 & 0 \\
        0 & 0 & 0 & 1
    \end{bsmallmatrix}}"] \& \overset{\C^4}{\bullet} \ar[ur, phantom,yshift=9, "\scalebox{1.5}{$\circlearrowleft$}"] \ar[r, "{\begin{bsmallmatrix}
        1 & 0 & 0 & 0 \\
        0 & 0 & 0 & 0 \\
        0 & 0 & 1 & 0 \\
        0 & 0 & 0 & 1
    \end{bsmallmatrix}}"] \& \overset{\C^4}{\bullet}
\end{tikzcd}.
\end{equation*}
The decomposition of $M^f$ is then $M^f=U^{(4,0,0)}\oplus  U^{(3,3,0)}\oplus U^{(2,0,0)}\oplus U^{(1,1,1)}\oplus U^{(0,4,4)}\oplus U^{(0,2,2)}\oplus U^{(0,0,3)}$.
\end{example}

\begin{remark}
    We denote the (unique) projective, injective and simple indecomposable representations of $(\Gamma,I)$ (see Definition \ref{def:projinj}) by $P(i,j), I(i,j)$ and $S(i,j)$, respectively.
    Among them, the ones of the form $U^{(h_1,\dots,h_n)}$, for some $h_1,\dots,h_n$, are $P(1,j), P(i,1), I(n+1,j)$ and $S(n+1,j)$.
    To simplify notation, we show them in a few examples for the specific case of $n+1=4$:
\begin{equation*}
P(1,1)=\begin{tikzcd}[]
\overset{\C}{\bullet} \ar[r, "\id"] \ar[d, "\id"]  & \overset{\C}{\bullet} \ar[r, "\id"] \ar[d, "\id"] & \overset{\C}{\bullet} \ar[d, "\id"] \\
\overset{\C}{\bullet} \ar[ur, phantom, "\scalebox{1.5}{$\circlearrowleft$}"] \ar[r, "\id"] \ar[d, "\id"] & \overset{\C}{\bullet} \ar[ur, phantom, "\scalebox{1.5}{$\circlearrowleft$}"] \ar[r, "\id"] \ar[d, "\id"]  & \overset{\C}{\bullet} \ar[d, "\id"]\\
\overset{\C}{\bullet} \ar[ur, phantom, "\scalebox{1.5}{$\circlearrowleft$}"] \ar[r, "\id"] \ar[d, "\id"] & \overset{\C}{\bullet} \ar[ur, phantom, "\scalebox{1.5}{$\circlearrowleft$}"] \ar[r, "\id"] \ar[d, "\id"] & \overset{\C}{\bullet} \ar[d, "\id"]\\
\overset{\C}{\bullet} \ar[ur, phantom, "\scalebox{1.5}{$\circlearrowleft$}"] \ar[r, "\id"] & \overset{\C}{\bullet} \ar[ur, phantom, "\scalebox{1.5}{$\circlearrowleft$}"] \ar[r, "\id"] & \overset{\C}{\bullet}\\
\end{tikzcd},\quad
P(1,2)= \begin{tikzcd}[]
\overset{0}{\bullet} \ar[r, "0"] \ar[d, "0"]  & \overset{\C}{\bullet} \ar[r, "\id"] \ar[d, "\id"] & \overset{\C}{\bullet} \ar[d, "\id"] \\
\overset{0}{\bullet} \ar[ur, phantom, "\scalebox{1.5}{$\circlearrowleft$}"] \ar[r, "0"] \ar[d, "0"] & \overset{\C}{\bullet} \ar[ur, phantom, "\scalebox{1.5}{$\circlearrowleft$}"] \ar[r, "\id"] \ar[d, "\id"]  & \overset{\C}{\bullet} \ar[d, "\id"]\\
\overset{0}{\bullet} \ar[ur, phantom, "\scalebox{1.5}{$\circlearrowleft$}"] \ar[r, "0"] \ar[d, "0"] & \overset{\C}{\bullet} \ar[ur, phantom, "\scalebox{1.5}{$\circlearrowleft$}"] \ar[r, "\id"] \ar[d, "\id"] & \overset{\C}{\bullet} \ar[d, "\id"]\\
\overset{0}{\bullet} \ar[ur, phantom, "\scalebox{1.5}{$\circlearrowleft$}"] \ar[r, "0"] & \overset{\C}{\bullet} \ar[ur, phantom, "\scalebox{1.5}{$\circlearrowleft$}"] \ar[r, "\id"] & \overset{\C}{\bullet}\\
\end{tikzcd},
\end{equation*}

\begin{equation*}
I(4,1)=\begin{tikzcd}[]
\overset{\C}{\bullet} \ar[r, "0"] \ar[d, "\id"]  & \overset{0}{\bullet} \ar[r, "0"] \ar[d, "0"] & \overset{0}{\bullet} \ar[d, "0"] \\
\overset{\C}{\bullet} \ar[ur, phantom, "\scalebox{1.5}{$\circlearrowleft$}"] \ar[r, "0"] \ar[d, "\id"] & \overset{0}{\bullet} \ar[ur, phantom, "\scalebox{1.5}{$\circlearrowleft$}"] \ar[r, "0"] \ar[d, "0"]  & \overset{0}{\bullet} \ar[d, "0"]\\
\overset{\C}{\bullet} \ar[ur, phantom, "\scalebox{1.5}{$\circlearrowleft$}"] \ar[r, "0"] \ar[d, "\id"] & \overset{0}{\bullet} \ar[ur, phantom, "\scalebox{1.5}{$\circlearrowleft$}"] \ar[r, "0"] \ar[d, "0"] & \overset{0}{\bullet} \ar[d, "0"]\\
\overset{\C}{\bullet} \ar[ur, phantom, "\scalebox{1.5}{$\circlearrowleft$}"] \ar[r, "0"] & \overset{0}{\bullet} \ar[ur, phantom, "\scalebox{1.5}{$\circlearrowleft$}"] \ar[r, "0"] & \overset{0}{\bullet}\\
\end{tikzcd},\quad
I(4,2)=\begin{tikzcd}[]
\overset{\C}{\bullet} \ar[r, "\id"] \ar[d, "\id"]  & \overset{\C}{\bullet} \ar[r, "0"] \ar[d, "\id"] & \overset{0}{\bullet} \ar[d, "0"] \\
\overset{\C}{\bullet} \ar[ur, phantom, "\scalebox{1.5}{$\circlearrowleft$}"] \ar[r, "\id"] \ar[d, "\id"] & \overset{\C}{\bullet} \ar[ur, phantom, "\scalebox{1.5}{$\circlearrowleft$}"] \ar[r, "0"] \ar[d, "\id"]  & \overset{0}{\bullet} \ar[d, "0"]\\
\overset{\C}{\bullet} \ar[ur, phantom, "\scalebox{1.5}{$\circlearrowleft$}"] \ar[r, "\id"] \ar[d, "\id"] & \overset{\C}{\bullet} \ar[ur, phantom, "\scalebox{1.5}{$\circlearrowleft$}"] \ar[r, "0"] \ar[d, "\id"] & \overset{0}{\bullet} \ar[d, "0"]\\
\overset{\C}{\bullet} \ar[ur, phantom, "\scalebox{1.5}{$\circlearrowleft$}"] \ar[r, "\id"] & \overset{\C}{\bullet} \ar[ur, phantom, "\scalebox{1.5}{$\circlearrowleft$}"] \ar[r, "0"] & \overset{0}{\bullet}\\
\end{tikzcd},
\end{equation*}

\begin{equation*}
I(4,3)=\begin{tikzcd}[]
\overset{\C}{\bullet} \ar[r, "\id"] \ar[d, "\id"]  & \overset{\C}{\bullet} \ar[r, "\id"] \ar[d, "\id"] & \overset{\C}{\bullet} \ar[d, "\id"] \\
\overset{\C}{\bullet} \ar[ur, phantom, "\scalebox{1.5}{$\circlearrowleft$}"] \ar[r, "\id"] \ar[d, "\id"] & \overset{\C}{\bullet} \ar[ur, phantom, "\scalebox{1.5}{$\circlearrowleft$}"] \ar[r, "\id"] \ar[d, "\id"]  & \overset{\C}{\bullet} \ar[d, "\id"]\\
\overset{\C}{\bullet} \ar[ur, phantom, "\scalebox{1.5}{$\circlearrowleft$}"] \ar[r, "\id"] \ar[d, "\id"] & \overset{\C}{\bullet} \ar[ur, phantom, "\scalebox{1.5}{$\circlearrowleft$}"] \ar[r, "\id"] \ar[d, "\id"] & \overset{\C}{\bullet} \ar[d, "\id"]\\
\overset{\C}{\bullet} \ar[ur, phantom, "\scalebox{1.5}{$\circlearrowleft$}"] \ar[r, "\id"] & \overset{\C}{\bullet} \ar[ur, phantom, "\scalebox{1.5}{$\circlearrowleft$}"] \ar[r, "\id"] & \overset{\C}{\bullet}\\
\end{tikzcd},\quad
S(4,2)= \begin{tikzcd}[]
\overset{0}{\bullet} \ar[r, "0"] \ar[d, "0"]  & \overset{0}{\bullet} \ar[r, "0"] \ar[d, "0"] & \overset{0}{\bullet} \ar[d, "0"] \\
\overset{0}{\bullet} \ar[ur, phantom, "\scalebox{1.5}{$\circlearrowleft$}"] \ar[r, "0"] \ar[d, "0"] & \overset{0}{\bullet} \ar[ur, phantom, "\scalebox{1.5}{$\circlearrowleft$}"] \ar[r, "0"] \ar[d, "0"]  & \overset{0}{\bullet} \ar[d, "0"]\\
\overset{0}{\bullet} \ar[ur, phantom, "\scalebox{1.5}{$\circlearrowleft$}"] \ar[r, "0"] \ar[d, "0"] & \overset{0}{\bullet} \ar[ur, phantom, "\scalebox{1.5}{$\circlearrowleft$}"] \ar[r, "0"] \ar[d, "0"] & \overset{0}{\bullet} \ar[d, "0"]\\
\overset{0}{\bullet} \ar[ur, phantom, "\scalebox{1.5}{$\circlearrowleft$}"] \ar[r, "0"] & \overset{\C}{\bullet} \ar[ur, phantom, "\scalebox{1.5}{$\circlearrowleft$}"] \ar[r, "0"] & \overset{0}{\bullet}\\
\end{tikzcd},
\end{equation*}

that is, $P(1,1)=U^{(4,4,4)}, P(1,2)=U^{(0,4,4)}, I(4,1)=U^{(4,0,0)}, I(4,2)=U^{(4,4,0)},$ $ I(4,3)=U^{(4,4,4)}$ and $S(4,2)=U^{(0,1,0)}$.
Some examples of injective and simple indecomposables of $(\Gamma,I)$ that are not of the form $U^{(h_1,h_2,h_3)}$ are the following:

\begin{equation*}
\begin{comment}
P(2,2)=\begin{tikzcd}[]
\overset{0}{\bullet} \ar[r, "0"] \ar[d, "0"]  & \overset{0}{\bullet} \ar[r, "0"] \ar[d, "0"] & \overset{0}{\bullet} \ar[d, "0"] \\
\overset{0}{\bullet} \ar[ur, phantom, "\scalebox{1.5}{$\circlearrowleft$}"] \ar[r, "0"] \ar[d, "0"] & \overset{\C}{\bullet} \ar[ur, phantom, "\scalebox{1.5}{$\circlearrowleft$}"] \ar[r, "\id"] \ar[d, "\id"]  & \overset{\C}{\bullet} \ar[d, "\id"]\\
\overset{0}{\bullet} \ar[ur, phantom, "\scalebox{1.5}{$\circlearrowleft$}"] \ar[r, "0"] \ar[d, "0"] & \overset{\C}{\bullet} \ar[ur, phantom, "\scalebox{1.5}{$\circlearrowleft$}"] \ar[r, "\id"] \ar[d, "\id"] & \overset{\C}{\bullet} \ar[d, "\id"]\\
\overset{0}{\bullet} \ar[ur, phantom, "\scalebox{1.5}{$\circlearrowleft$}"] \ar[r, "0"] & \overset{\C}{\bullet} \ar[ur, phantom, "\scalebox{1.5}{$\circlearrowleft$}"] \ar[r, "\id"] & \overset{\C}{\bullet}\\
\end{tikzcd}, \quad
\end{comment}
I(2,2)= \begin{tikzcd}[]
\overset{\C}{\bullet} \ar[r, "\id"] \ar[d, "\id"]  & \overset{\C}{\bullet} \ar[r, "\id"] \ar[d, "\id"] & \overset{0}{\bullet} \ar[d, "0"] \\
\overset{\C}{\bullet} \ar[ur, phantom, "\scalebox{1.5}{$\circlearrowleft$}"] \ar[r, "\id"] \ar[d, "0"] & \overset{\C}{\bullet} \ar[ur, phantom, "\scalebox{1.5}{$\circlearrowleft$}"] \ar[r, "0"] \ar[d, "0"]  & \overset{0}{\bullet} \ar[d, "0"]\\
\overset{0}{\bullet} \ar[ur, phantom, "\scalebox{1.5}{$\circlearrowleft$}"] \ar[r, "0"] \ar[d, "0"] & \overset{0}{\bullet} \ar[ur, phantom, "\scalebox{1.5}{$\circlearrowleft$}"] \ar[r, "0"] \ar[d, "0"] & \overset{0}{\bullet} \ar[d, "0"]\\
\overset{0}{\bullet} \ar[ur, phantom, "\scalebox{1.5}{$\circlearrowleft$}"] \ar[r, "0"] & \overset{0}{\bullet} \ar[ur, phantom, "\scalebox{1.5}{$\circlearrowleft$}"] \ar[r, "0"] & \overset{0}{\bullet}\\
\end{tikzcd}, \quad
S(3,2)=\begin{tikzcd}[]
\overset{0}{\bullet} \ar[r, "0"] \ar[d, "0"]  & \overset{0}{\bullet} \ar[r, "0"] \ar[d, "0"] & \overset{0}{\bullet} \ar[d, "0"] \\
\overset{0}{\bullet} \ar[ur, phantom, "\scalebox{1.5}{$\circlearrowleft$}"] \ar[r, "0"] \ar[d, "0"] & \overset{0}{\bullet} \ar[ur, phantom, "\scalebox{1.5}{$\circlearrowleft$}"] \ar[r, "0"] \ar[d, "0"]  & \overset{0}{\bullet} \ar[d, "0"]\\
\overset{0}{\bullet} \ar[ur, phantom, "\scalebox{1.5}{$\circlearrowleft$}"] \ar[r, "0"] \ar[d, "0"] & \overset{\C}{\bullet} \ar[ur, phantom, "\scalebox{1.5}{$\circlearrowleft$}"] \ar[r, "0"] \ar[d, "0"] & \overset{0}{\bullet} \ar[d, "0"]\\
\overset{0}{\bullet} \ar[ur, phantom, "\scalebox{1.5}{$\circlearrowleft$}"] \ar[r, "0"] & \overset{\C}{\bullet} \ar[ur, phantom, "\scalebox{1.5}{$\circlearrowleft$}"] \ar[r, "0"] & \overset{0}{\bullet}\\
\end{tikzcd}.
\end{equation*}
\end{remark}

\begin{example}\label{ex:decompM}
    We recalled in Section \ref{sec:qGandschubvar} the representation $M$,
\begin{comment}
    \begin{equation*}
\begin{tikzcd}[sep=large]
\overset{\C}{\bullet} \ar[r, "\id"] \ar[d, "\iota_{2,1}"]  & \overset{\C}{\bullet} \ar[r, "\id"] \ar[d, "\iota_{2,1}"] & ...\ar[r, "\id"] \ar[d, "\iota_{2,1}"] & \overset{\C}{\bullet} \ar[d, "\iota_{2,1}"] \\
\overset{\C^2}{\bullet} \ar[ur, phantom, "\scalebox{1.5}{$\circlearrowleft$}"] \ar[r, "\id"] \ar[d, "\iota_{3,2}"] & \overset{\C^2}{\bullet} \ar[ur, phantom, "\scalebox{1.5}{$\circlearrowleft$}"] \ar[r, "\id"] \ar[d, "\iota_{3,2}"]  & ... \ar[ur, phantom,xshift=5, "\scalebox{1.5}{$\circlearrowleft$}"] \ar[r, "\id"] \ar[d, "\iota_{3,2}"] & \overset{\C^2}{\bullet} \ar[d, "\iota_{3,2}"] \\
... \ar[ur, phantom,xshift=5, "\scalebox{1.5}{$\circlearrowleft$}"] \ar[r, "\id"] \ar[d, "\iota_{n+1,n}"] & ... \ar[ur, phantom,xshift=5, "\scalebox{1.5}{$\circlearrowleft$}"] \ar[r, "\id"] \ar[d, "\iota_{n+1,n}"] & ... \ar[ur, phantom, xshift=5, "\scalebox{1.5}{$\circlearrowleft$}"] \ar[r, "\id"] \ar[d,yshift=5, "\iota_{n+1,n}"] & ... \ar[d, "\iota_{n+1,n}"]\\
\overset{\C^{n+1}}{\bullet} \ar[ur, phantom, "\scalebox{1.5}{$\circlearrowleft$}"] \ar[r, "\id"] & \overset{\C^{n+1}}{\bullet} \ar[ur, phantom, "\scalebox{1.5}{$\circlearrowleft$}"] \ar[r, "\id"] & ... \ar[ur, phantom, xshift=5,yshift=5, "\scalebox{1.5}{$\circlearrowleft$}"] \ar[r, "\id"] & \overset{\C^{n+1}}{\bullet}\\
\end{tikzcd}
\end{equation*}
\end{comment}
which is the element of $R^{\iota}_{\bf d}$ that corresponds to the choice $f_{n+1}^1=f_{n+1}^2=\dots=f_{n+1}^{n-1}=\id$. Its decomposition is $M= U^{(n+1,\dots,n+1)} \oplus U^{(n,\dots,n)} \oplus U^{(n-1,\dots,n-1)} \oplus \dots \oplus U^{(2,\dots,2)} \oplus U^{(1,\dots,1)} = P(1,1)\oplus P(2,1) \oplus \dots \oplus P(n,1) \oplus P(n+1,1)$. Since it is a direct sum of projective $(\Gamma,I)$-representations, $M$ is a projective $(\Gamma,I)$-representation as well.
\end{example}

\begin{theorem}\label{thm:indecompU}
    All representations in $R^{\iota}_{\bf d}$ can be decomposed as direct sums of the indecomposable $(\Gamma,I)$-representations $U^{(h_1,\dots,h_n)}$.
\end{theorem}
\begin{proof}
    First, we observe that the $(\Gamma,I)$-representations $U^{(h_1,\dots,h_n)}$ are indecomposable, because they are thin and all linear maps connecting two one-dimensional vector spaces are identity maps (the only proper subspace of a one-dimensional vector space is the trivial subspace, which is the domain or codomain of only the zero map).
    To show that all representations in $R^{\iota}_{\bf d}$ can be decomposed as their direct sums, we recall from \eqref{eq:ftuple} that each representation $M^f$ in $R^{\iota}_{\bf d}$ is determined by the choice of $f=(f_{n+1}^1,\dots,f_{n+1}^{n-1})$ in $\prod_{j=1}^{n-1} U_{n+1}$: every other linear map $f_{i}^j$ corresponding to the horizontal arrow $(i,j)\to (i,j+1)$ of $(\Gamma, I)$ is represented by the appropriate submatrix of $f_{n+1}^j$ (i.e., it is the restriction of $f_{n+1}^j$ to $\C^i$), while all maps corresponding to vertical arrows of $(\Gamma, I)$ are fixed as the standard inclusions $\C^i \hookrightarrow \C^{i+1}$.
    In other words, any indecomposable representation appearing in the decomposition of $M^f$ must be a representation of $(\Gamma, I)$ whose linear maps associated to the vertical arrows are either zero or identity maps and such that, if the linear map associated to a horizontal arrow $(i,j)\to (i,j+1)$ is zero, then all linear maps associated to arrows of the form $(i,k)\to (i,k+1)$ for $k>l$ are zero.
    Then, we consider the subquivers of $(\Gamma,I)$ that start at $(i,1)$ and end at $ (i,n)$, for $i=1,\dots,n+1$ (i.e., each row of $(\Gamma,I)$ is considered as a subquiver). These are equioriented Dynkin quivers of type $\mathbb{A}_n$, whose indecomposables $U_{a,b}$ are thin and given by connected intervals $[a,b]$ with $1\leq a \leq b \leq n$.
    The statement then follows by applying this description of the indecomposables $U_{a,b}$ to all the considered subquivers of $(\Gamma,I)$, starting from the topmost row downwards; the definition of each $f_j^i$ as the restriction of $f_{n+1}^j$ to $\C^i$ ensures the commutativity of all arising square diagrams.
\end{proof}

Our goal now is to find a way to parametrise the representations in $R^{\iota}_{\bf d}$ or, more precisely, to parametrise their isomorphism classes. We recall from Lemma \ref{lemma:isomGL} that the isomorphism classes of the representations of a given quiver coincide with the orbits under the action of $ G_{\bf d}=\prod_{i\in Q_0} \GL_{d_i}(\K)$, defined on the variety $R_{\bf d}$ as
\begin{equation*}
    g \op M= g_{t(\alpha)} M_{\alpha} g^{-1}_{s(\alpha)}
\end{equation*}
for $g \coloneqq (g_i)_{i\in Q_0} \in G_{\bf d}$. The action of $G_{\bf d}$, however, is not compatible with the restrictions we impose on the $(\Gamma,I)$-representations when we consider the subvariety $R^{\iota}_{\bf d}$ of $R_{\bf d}$: for $M^f \in R^{\iota}_{\bf d}$, the representation $g \op M^f \in R_{\bf d}$ is isomorphic to $M^f$ but is not, in general, an element of $R^{\iota}_{\bf d}$.
This is because the linear maps associated to the vertical arrows of $g \op M^f$ are not, in general, of the form $\iota_{i+1,i}$, the standard inclusion of $\C^i$ into $\C^{i+1}$ (see Definition \ref{def:subvariety} for the conditions defining $R^{\iota}_{\bf d}$). For this reason, we consider instead the action of the maximal subgroup of $G_{\bf d}$ that is compatible with the definition of the subvariety $R^{\iota}_{\bf d}$:

\begin{definition}\label{def:Baction}
    Given a representation $M^f = (f_{n+1}^1,\dots,f_{n+1}^{n-1}) \in R^{\iota}_{\bf d}$, we denote by $\orb^{\iota}_{M^f}$ the orbit of $M^f$ under the action of $G^{\iota}_{\bf d}\coloneqq \prod_{i=1}^{n} B_{n+1}$, where $B_{n+1}$ is the group of invertible upper-triangular matrices of size $n+1$, defined as
    \begin{equation*}
    h \op M^f= (h_2 f_{n+1}^1 h_1^{-1}, h_3 f_{n+1}^2 h_2^{-1}, \dots, h_n f_{n+1}^{n-1} h_{n-1}^{-1})
\end{equation*}
for some $h \in G^{\iota}_{\bf d}$.
\end{definition}

Now we can reformulate our goal: we want to parametrise the orbits $\orb^{\iota}_{M^f}$ under the action of $G^{\iota}_{\bf d}$ or, equivalently, we want to parametrise the classes of representations in $R^{\iota}_{\bf d}$ that are isomorphic via a morphism in $\prod_{i=1}^{n} B_{n+1}$. We denote such isomorphisms by
\begin{equation}
    M^f \stackrel{B}{\cong} M^{g},
\end{equation}
and say that $f$ and $g$ are in the same $B$-isomorphism class, or $B$-orbit, of representations in $R^{\iota}_{\bf d}$.

\begin{remark}\label{remark:actionrestriction}
As shown in Remark \ref{remark:subvariety}, the linear maps $f_{i}^j$ of a representation $M^f = (f_{n+1}^1,\dots,f_{n+1}^{n-1}) \in R^{\iota}_{\bf d}$, for $i<n+1$, are the restrictions of $f_{n+1}^j$ to $\C^i$.
This means that, if we denote $M^{g}\coloneqq h\op M^f$ for some $h=(h_1,\dots,h_n) \in G^{\iota}_{\bf d}$, the linear maps $g_{i}^j$ of $M^{g}$ for $i<n+1$ are the restrictions of $g_{n+1}^j$ to $\C^i$.
In other words, the action of $h$ on $M^f$ on the linear maps $(f_{i}^1, \dots,f_{i}^{n-1})$, for a fixed $i<n+1$, is given by the action of the restrictions to $\C^i$ of the fixed maps $(h_1,\dots,h_n)$ on $(f_i^1, \dots,f_i^{n-1})$. We clarify this through the following example.
\end{remark}

\begin{example}\label{ex:actiondim3}
Consider the quiver $(\Gamma,I)$ for $n+1=3$ and its representation $M^f$ for some $f=(f_{3}^1)=\begin{bsmallmatrix}
        a & b & d\\
        0 & c & e\\
        0 & 0 & l
    \end{bsmallmatrix}\in U_{3}$:
\begin{equation*}
M^f:
\begin{tikzcd}[sep=large, ampersand replacement=\&]
\overset{\C}{\bullet} \ar[r, "\begin{bsmallmatrix}
        a
    \end{bsmallmatrix}"] \ar[d, "\iota_{2,1}"]  \& \overset{\C}{\bullet} \ar[d, "\iota_{2,1}"] \\
\overset{\C^2}{\bullet} \ar[ur,phantom, "\scalebox{1.5}{$\circlearrowleft$}"] \ar[r, "{\begin{bsmallmatrix}
        a & b\\
        0 & c
    \end{bsmallmatrix}}"] \ar[d, "\iota_{3,2}"] \& \overset{\C^2}{\bullet}  \ar[d, "\iota_{3,2}"] \\
\overset{\C^3}{\bullet} \ar[ur, phantom,yshift=5, "\scalebox{1.5}{$\circlearrowleft$}"] \ar[r, "{\begin{bsmallmatrix}
        a & b & d\\
        0 & c & e\\
        0 & 0 & l
    \end{bsmallmatrix}}"] \& \overset{\C^3}{\bullet}
\end{tikzcd}.
\end{equation*}
The orbit of $M^f$ under the action of $G^{\iota}_{\bf d}=B_3\times B_3$ is
\begin{equation*}
    \orb^{\iota}_{M^f}=\{h_2 f_{3}^1 h_1^{-1} | (h_1,h_2)\in B_3\times B_3 \}
\end{equation*}
and, if we denote $h_1=\begin{bsmallmatrix}
        x_1 & x_2 & x_3\\
        0 & x_4 & x_5\\
        0 & 0 & x_6
    \end{bsmallmatrix}$ and $h_2=\begin{bsmallmatrix}
        y_1 & y_2 & y_3\\
        0 & y_4 & y_5\\
        0 & 0 & y_6
    \end{bsmallmatrix}$, then the representation $M^{g}\coloneqq h \op M^f$ is
    \begin{equation*}
    M^{g}=
\begin{tikzcd}[sep=large, ampersand replacement=\&]
\overset{\C}{\bullet} \ar[r, "g_{1}^1"] \ar[d, "\iota_{2,1}"]  \& \overset{\C}{\bullet} \ar[d, "\iota_{2,1}"] \\
\overset{\C^2}{\bullet} \ar[ur,phantom, "\scalebox{1.5}{$\circlearrowleft$}"] \ar[r, "g_{2}^1"] \ar[d, "\iota_{3,2}"] \& \overset{\C^2}{\bullet}  \ar[d, "\iota_{3,2}"] \\
\overset{\C^3}{\bullet} \ar[ur, phantom,yshift=5, "\scalebox{1.5}{$\circlearrowleft$}"] \ar[r,"g_{3}^1"] \& \overset{\C^3}{\bullet}
\end{tikzcd}
\end{equation*}
where
\begin{equation*}
g_{1}^1=\begin{bsmallmatrix}
        y_1
    \end{bsmallmatrix}\begin{bsmallmatrix}
        a
    \end{bsmallmatrix}\begin{bsmallmatrix}
        x_1
    \end{bsmallmatrix}^{-1},\quad g_{2}^1=\begin{bsmallmatrix}
        y_1 & y_2\\
        0 & y_4 
    \end{bsmallmatrix}\begin{bsmallmatrix}
        a & b\\
        0 & c
    \end{bsmallmatrix}\begin{bsmallmatrix}
        x_1 & x_2\\
        0 & x_4 \\
    \end{bsmallmatrix}^{-1},
    \end{equation*}
    \begin{equation*}
    g_{3}^1=\begin{bsmallmatrix}
        y_1 & y_2 & y_3\\
        0 & y_4 & y_5\\
        0 & 0 & y_6
    \end{bsmallmatrix}\begin{bsmallmatrix}
        a & b & d\\
        0 & c & e\\
        0 & 0 & l
    \end{bsmallmatrix}\begin{bsmallmatrix}
        x_1 & x_2 & x_3\\
        0 & x_4 & x_5\\
        0 & 0 & x_6
    \end{bsmallmatrix}^{-1}.
    \end{equation*}
\end{example}

\begin{remark}\label{remark:effectofBaction}
    The row (or column) echelon form of a matrix $f\in U_{n+1}$ can be obtained via the action of $h\in B_{n+1}\times B_{n+1}$ given, as in Definition \ref{def:Baction}, by $h\op f = h_2 f h_1^{-1}$: the effect of this action is ``sweeping upwards'' and ``sweeping to the right'' of the pivots of $f$, allowing us to transform $f$ simultaneously into its reduced row and column echelon form. We consider then this form as the standard representative for the orbit of $f$: it is an upper-triangular matrix whose entries are all equal to 0, except for at most one entry equal to 1 in each row and column. Such matrices (in general, not upper-triangular) are known as \Def{partial permutation matrices}, a term employed, for instance, in \cite{knutson2005grobner} and \cite{miller2005matrix}.
\end{remark}

\begin{example}\label{ex:15reps}
    For $n+1=3$, the standard representatives of the orbits under the action given in Definition \ref{def:Baction} are the following matrices:
    \begin{equation*}
        f^1=\begin{bsmallmatrix}
        1 & 0 & 0 \\
        0 & 1 & 0 \\
        0 & 0 & 1
    \end{bsmallmatrix}, f^2=\begin{bsmallmatrix}
        1 & 0 & 0 \\
        0 & 1 & 0 \\
        0 & 0 & 0
    \end{bsmallmatrix}, f^3=\begin{bsmallmatrix}
        1 & 0 & 0 \\
        0 & 0 & 0 \\
        0 & 0 & 1
    \end{bsmallmatrix}, f^4=\begin{bsmallmatrix}
        0 & 0 & 0 \\
        0 & 1 & 0 \\
        0 & 0 & 1
    \end{bsmallmatrix}, f^5=\begin{bsmallmatrix}
        1 & 0 & 0 \\
        0 & 0 & 1 \\
        0 & 0 & 0
    \end{bsmallmatrix},
    \end{equation*}
    \begin{equation*}
    f^6=\begin{bsmallmatrix}
        0 & 0 & 1 \\
        0 & 1 & 0 \\
        0 & 0 & 0
    \end{bsmallmatrix},
     f^7=\begin{bsmallmatrix}
        0 & 1 & 0 \\
        0 & 0 & 1 \\
        0 & 0 & 0
    \end{bsmallmatrix},
    f^8=\begin{bsmallmatrix}
        0 & 1 & 0 \\
        0 & 0 & 0 \\
        0 & 0 & 1
    \end{bsmallmatrix},
     f^9=\begin{bsmallmatrix}
        1 & 0 & 0 \\
        0 & 0 & 0 \\
        0 & 0 & 0
    \end{bsmallmatrix},
     f^{10}=\begin{bsmallmatrix}
        0 & 0 & 0 \\
        0 & 1 & 0 \\
        0 & 0 & 0
    \end{bsmallmatrix},
    \end{equation*}
    \begin{equation*}
        f^{11}=\begin{bsmallmatrix}
        0 & 0 & 0 \\
        0 & 0 & 0 \\
        0 & 0 & 1
    \end{bsmallmatrix},
    f^{12}=\begin{bsmallmatrix}
        0 & 0 & 0 \\
        0 & 0 & 1 \\
        0 & 0 & 0
    \end{bsmallmatrix},
    f^{13}=\begin{bsmallmatrix}
        0 & 0 & 1 \\
        0 & 0 & 0 \\
        0 & 0 & 0
    \end{bsmallmatrix},
    f^{14}=\begin{bsmallmatrix}
        0 & 1 & 0 \\
        0 & 0 & 0 \\
        0 & 0 & 0
    \end{bsmallmatrix},  
    f^{15}=\begin{bsmallmatrix}
        0 & 0 & 0 \\
        0 & 0 & 0 \\
        0 & 0 & 0
    \end{bsmallmatrix}.
    \end{equation*}
    These represent the 15 orbits of the form
\begin{equation*}
    \orb^{\iota}_{M^f}=\{h_2 f_{3}^1 h_1^{-1} | (h_1,h_2)\in B_3\times B_3 \},
\end{equation*}
where $f_{3}^1\in U_3$ is the linear map defining the representation $M^f$ of Example \ref{ex:actiondim3}.
\end{example}

\begin{remark}\label{remark:numberoforbits}
     Because of the definition of $(\Gamma, I)$ (see Section \ref{sec:qGandschubvar}), considering orbits of the form given in \ref{ex:15reps} only makes sense in dimension $n+1=3$: for $n+1=2$ the quiver is not defined and, for $n+1>3$, the $(\Gamma, I)$-representations consist of sequences of more than one linear map. However, computing how many such orbits exist for a generic dimension $n+1$ is straightforward: the number of orbits $\orb^{\iota}_{M^f}$ in dimension $n+1$, where $M^f$ is determined by one $n+1 \times n+1$ upper-triangular matrix, is the $n+2$-th Bell number, defined recursively by
\begin{equation*}
    b_{n+2}=\sum_{k=0}^{n+1} \binom{n+1}{k} b_{k}.
\end{equation*}
This number is obtained by counting (recursively) the possible configurations of pivots of an $n+1 \times n+1$ upper-triangular matrix of any rank between 0 and $n+1$.
\end{remark}

\subsection[Two parametrisations]{Two parametrisations for the representations in $R^{\iota}_{\bf d}$}\label{sec:parametrisations}

Differently from Example \ref{ex:15reps}, understanding and counting the orbits $\orb^{\iota}_{M^f}$ in a generic dimension is a more complex question: the base changes performed by the action are not independent of one another, because the linear maps of a representation $M^f$ share source and target vertices.

In order to prove that there are finitely many orbits $\orb^{\iota}_{M^f}$ under the action given in Definition \ref{def:Baction} and to parametrise such orbits, we recall a result for the Dynkin quivers of type $\mathbb{A}_m$ (with any orientation) presented in \cite{abeasis1985degenerations}.
Let $Q_m$ be a quiver of type $\mathbb{A}_m$. Consider a $Q_m$-representation $A=(A_1,\dots,A_{m-1})$ and any pair of indices $u,v$ such that $1\leq u \leq v\leq m$. Let $\varphi^A_{uv}$ denote the linear map going from the direct sum of the vector spaces relative to all the sources to the one relative to all the sinks between $u$ and $v$ (included $u$ and $v$) whose components are
\begin{equation*}
\begin{aligned}
 V_{s_{t-1}}\oplus V_{s_{t+1}} & \to V_{s_t}\\
 (z,z') & \mapsto (\Bar{A}_{t-1,t}(z)-\Bar{A}_{t+1,t}(z'))
\end{aligned}
\end{equation*}
where $\Bar{A}_{pt}$, for $p=t-1,t+1$, is the composition of all the maps $A_i$ going from the sources $s_{t-1}$ or $s_{t+1}$ to the sink $s_t$.

\begin{theorem}\cite[Proposition 2.7]{abeasis1985degenerations}\label{thm:paramtypeA}
    The orbits of the $Q_m$-representations $A$ under the action of $G_{\bf d}=\prod_{i=1}^m \GL_{d_i}(\K)$ are parametrised by the sets of non-negative integers $N^A=\{N^A_{uv} \}_{1\leq u\leq v \leq m}$ defined as
    \begin{equation*}
    \begin{cases}
         N^A_{uv}\coloneqq \rk \varphi^A_{uv} & \text{ if } u<v\\
         N^A_{uv}\coloneqq \dim V_u & \text{ if } u=v
    \end{cases}.
    \end{equation*}
\end{theorem}

Let us now return to the question of parametrising the orbits $\orb^{\iota}_{M^f}$ under the action given in Definition \ref{def:Baction}. First, we define the following sequence of non-negative integers:

\begin{definition}\label{def:rankvector}
    Given a representation $M^f \in R^{\iota}_{\bf d}$, the \Def{rank vector} of $M^f$ is ${\bf r}^f\coloneqq ((r^f_{l,s}),(r_{ij_1j_2k}^f))$, for $l,i=1,\dots,n+1$, $s=1,\dots,n$, $k=1,\dots,i+1$ and $1\leq j_1 \leq j_2 \leq n-1$, where $r^f_{l,s}\coloneqq \dim(M^f_{l,s})$ and
    \begin{equation}
        \begin{cases}
            r_{ij_1j_2k}^f \coloneqq \dim(\im (f_{i}^{j_2}\circ\cdots\circ f_{i}^{j-1}) \cap \im(\iota_{i,i-1}\circ\dots\circ\iota_{k,k-1})) \text{ for } k=1,\dots,i \\
            r_{ij_1j_2k}^f \coloneqq \dim(\im (f_{i}^{j_2}\circ\cdots\circ f_{i}^{j-1}) \cap \C^k) \text{ for } k=i+1
        \end{cases}.
    \end{equation}
\end{definition}

In words, the first part of ${\bf r}^f$ is the dimension vector of $M^f$, while the information encoded in the second part of the rank vector of $M^f$ consists of the dimensions of all possible intersections of the images of the linear maps that determine $M^f$. When $k=i+1$, the number $r_{ij_1j_2k}^f$ is exactly the rank of $f_{i}^{j_2}\circ\cdots\circ f_{i}^{j-1}$.

\begin{example}\label{ex:map011}
    The rank vector associated to the following $(\Gamma,I)$-representation:
    \begin{equation*}
M^f:
\begin{tikzcd}[sep=large, ampersand replacement=\&]
\overset{\C}{\bullet} \ar[r, "\begin{bsmallmatrix}
        0
    \end{bsmallmatrix}"] \ar[d, "\iota_{2,1}"]  \& \overset{\C}{\bullet} \ar[d, "\iota_{2,1}"] \\
\overset{\C^2}{\bullet} \ar[ur,phantom, "\scalebox{1.5}{$\circlearrowleft$}"] \ar[r, "{\begin{bsmallmatrix}
        0 & 0\\
        0 & 1
    \end{bsmallmatrix}}"] \ar[d, "\iota_{3,2}"] \& \overset{\C^2}{\bullet}  \ar[d, "\iota_{3,2}"] \\
\overset{\C^3}{\bullet} \ar[ur, phantom,yshift=5, "\scalebox{1.5}{$\circlearrowleft$}"] \ar[r, "{\begin{bsmallmatrix}
        0 & 0 & 0\\
        0 & 1 & 0\\
        0 & 0 & 1
    \end{bsmallmatrix}}"] \& \overset{\C^3}{\bullet}
\end{tikzcd}
\end{equation*}
is ${\bf r}^f=(\dim(\im (f_{1}^1) \cap \C^1), \dim(\im (f_{2}^1) \cap \C^1),\dim(\im (f_{2}^1) \cap \C^2),\dim(\im (f_{3}^1) \cap \C^1), \dim(\im (f_{3}^1) \cap \C^2), \dim(\im (f_{3}^1) \cap \C^3))=(0,0,1,0,1,2)$.
\end{example}

In order to show that the rank vectors parametrise the orbits $\orb^{\iota}_{M^f}$, we need the following technical Lemma:

\begin{lemma}\label{lemma:indipvectors}
    The rank vectors of the indecomposable $(\Gamma,I)$-representations of the form $U^{(h_1,\dots,h_n)}$, defined in \eqref{eq:indecompUspaces}, are independent as elements of the free $\Z$-module $\Z^N$, where $N$ denotes the number of entries of any rank vector in dimension $n+1$.
\end{lemma}
\begin{proof}
We proceed by induction on the number of columns in the support of the indecomposables. For the base case, we observe that the indecomposables of the form $U^{(h_1,\dots,h_n)}$ supported on only one column of $(\Gamma,I)$ can be regarded as the following indecomposables of the equioriented quiver of type $\mathbb{A}_n$:
\begin{equation*}
\begin{aligned}
\begin{tikzcd}[]
\overset{0}{\bullet} \ar[d, "0"] \\
\overset{0}{\bullet} \ar[d, "0"] \\
\dots \ar[d, "0"] \\
\overset{\C}{\bullet}
\end{tikzcd},\quad &
\begin{tikzcd}[]
\overset{0}{\bullet} \ar[d, "0"] \\
\dots \ar[d, "0"] \\
\overset{\C}{\bullet} \ar[d, "1"] \\
\overset{\C}{\bullet}
\end{tikzcd},\quad & \dots,\quad &
\begin{tikzcd}[]
\overset{0}{\bullet} \ar[d, "0"] \\
\overset{\C}{\bullet} \ar[d, "1"] \\
\dots \ar[d, "1"] \\
\overset{\C}{\bullet}
\end{tikzcd}, \quad &
\begin{tikzcd}[]
\overset{\C}{\bullet} \ar[d, "1"] \\
\dots \ar[d, "1"] \\
\overset{\C}{\bullet} \ar[d, "1"] \\
\overset{\C}{\bullet}
\end{tikzcd}
\end{aligned}.
\end{equation*}
The dimension vectors of such indecomposables form a basis for the dimension vectors of all the indecomposables $U^{(h_1,\dots,h_n)}$, and therefore their set of rank vectors is an independent set. For all other indecomposables - that is, the ones supported on more than one column - in order to prove the independence we need to consider the second part of the rank vectors as well.

For the induction step, we assume the claim to be true when the number of columns in the support of the considered indecomposable $V$ is at most $m-1$ (for some $m>1$) and want to prove it for the indecomposables obtained by adding one column to the support of $V$. Let us denote such indecomposables by $V'$.
We call $\hat{i}$ the number $n+1-h_{m-1}$, which is the first row in the $m-1$-th column where the dimension vector of $V$ is not zero. Because of the condition $h_a\leq h_b$ for every $a\leq b$, which holds for all indecomposables of the form $U^{(h_1,\dots,h_n)}$, we know that the first row in the $m$-th column where the dimension vector of $V'$ is not zero can be at most $\hat{i}$.
Now, if this row is exactly $\hat{i}$, then the rank vector of $V'$ is:
\begin{equation*}
\begin{aligned}
{\bf r}^{V'}=(& 0,\dots,0,r_{\hat{i},1,m-1,\hat{i}}^{V},0,\dots,0,r_{\hat{i},2,m-1,\hat{i}}^{V},0,\dots,0,\textcolor{red}{1},0,\dots,0,r_{\hat{i}+1,1,m-1,\hat{i}}^{V},\\
& r_{\hat{i}+1,1,m-1,\hat{i}+1}^{V},\dots,\textcolor{red}{1},0,\dots,r_{\hat{i}+2,1,m-1,\hat{i}}^{V},r_{\hat{i}+2,1,m-1,\hat{i}+1}^{V},r_{\hat{i}+2,1,m-1,\hat{i}+2}^{V},\dots,\\
& \textcolor{red}{1},0,\dots,1,\dots,1)
\end{aligned}
\end{equation*}
where the 1s in red correspond, respectively, to $r_{\hat{i},m-1,m,\hat{i}}^{V'}$, $r_{\hat{i}+1,m-1,m,\hat{i}}^{V'}$, $r_{\hat{i}+2,m-1,m,\hat{i}}^{V'}$ and so on until $r_{n+1,m-1,m,\hat{i}}^{V'}$.
Because of the definition of the index $\hat{i}$, the corresponding entries of the rank vector ${\bf r}^{V}$ (and of the rank vectors corresponding to indecomposables with support smaller than the support of $V$) were zero, meaning that the set of all rank vectors of the indecomposables considered so far is an independent set in $\Z^N$ when we add ${\bf r}^{V'}$ to it. Similarly, we know that the entries of ${\bf r}^{V'}$ from $r_{n+1,m-1,m,\hat{i}}^{V'}$ to $r_{n+1,m-1,m,n+1}^{V'}$ (the last entries) are ones.

If, instead, the first row in the $m$-th column where the dimension vector of $V'$ is not zero is $\hat{i}-1$, then the new 1s in the rank vector ${\bf r}^{V'}$ correspond to all entries labelled by $\ell,m-1,m,\ell'$ where $\ell\geq \hat{i}$ and $\hat{i}-1\leq \ell'<\hat{i}$, and therefore the set of all rank vectors of the indecomposables considered so far plus ${\bf r}^{V'}$ is an independent set. We proceed analogously if the first row in the $m$-th column where the dimension vector of $V'$ is not zero is $\hat{i}-2$: then the new 1s in ${\bf r}^{V'}$ are the entries labelled by $\ell,m-1,m,\ell'$ for $\ell\geq \hat{i}$ and $\hat{i}-2\leq \ell'<\hat{i}$, and so forth until the first row in the $m$-th column where the dimension vector of $V'$ is not zero is the first row of $(\Gamma,I)$. Adding all such rank vectors to the preexisting set of rank vectors does not change its independency in $\Z^N$. This proves the statement in the induction step and concludes the proof.
\end{proof}

The following example illustrates the idea of the proof of Lemma \ref{lemma:indipvectors} in dimension $n+1=3$:

\begin{example}
The rank vectors of the 12 indecomposables of $(\Gamma,I)$ of the form $U^{h_1,h_2}$ are:

\begin{equation*}
\begin{aligned}
    {\bf r}^{U^1}=(1,1,1,1,1,1,1,1,1,1,1,1), \quad & {\bf r}^{U^2}=(1,1,1,0,0,0,0,0,0,0,0,0), \\
    {\bf r}^{U^3}=(0,0,0,1,1,1,0,0,0,0,0,0), \quad & {\bf r}^{U^4}=(0,1,1,0,0,0,0,0,0,0,0,0), \\
    {\bf r}^{U^5}=(0,0,0,0,1,1,0,0,0,0,0,0), \quad &  {\bf r}^{U^6}=(0,0,1,0,0,0,0,0,0,0,0,0), \\
    {\bf r}^{U^7}=(0,0,0,0,0,1,0,0,0,0,0,0), \quad & {\bf r}^{U^8}=(0,1,1,1,1,1,0,1,1,1,1,1), \\
    {\bf r}^{U^9}=(0,0,1,1,1,1,0,0,0,1,1,1), \quad & {\bf r}^{U^{10}}=(0,1,1,0,1,1,0,0,1,0,1,1), \\
    {\bf r}^{U^{11}}=(0,0,1,0,1,1,0,0,0,0,1,1), \quad & {\bf r}^{U^{12}}=(0,0,1,0,0,1,0,0,0,0,0,1).
\end{aligned},
\end{equation*}

where

\begin{equation*}
\begin{aligned}
& U^1=
\begin{tikzcd}[]
\overset{\C}{\bullet} \ar[r, "1"] \ar[d, "1"]  & \overset{\C}{\bullet} \ar[d, "1"] \\
\overset{\C}{\bullet} \ar[ur,phantom, "\scalebox{1.5}{$\circlearrowleft$}"] \ar[r, "1"] \ar[d, "1"] & \overset{\C}{\bullet}  \ar[d, "1"] \\
\overset{\C}{\bullet} \ar[ur, phantom, "\scalebox{1.5}{$\circlearrowleft$}"] \ar[r, "1"] & \overset{\C}{\bullet}
\end{tikzcd},\;
U^{2\phantom{0}}=
\begin{tikzcd}[]
\overset{\C}{\bullet} \ar[r, "0"] \ar[d, "1"]  & \overset{0}{\bullet} \ar[d, "0"] \\
\overset{\C}{\bullet} \ar[ur,phantom, "\scalebox{1.5}{$\circlearrowleft$}"] \ar[r, "0"] \ar[d, "1"] & \overset{0}{\bullet}  \ar[d, "0"] \\
\overset{\C}{\bullet} \ar[ur, phantom, "\scalebox{1.5}{$\circlearrowleft$}"] \ar[r, "0"] & \overset{0}{\bullet}
\end{tikzcd},\;
U^{3\phantom{0}}=
\begin{tikzcd}[]
\overset{0}{\bullet} \ar[r, "0"] \ar[d, "0"]  & \overset{\C}{\bullet} \ar[d, "1"] \\
\overset{0}{\bullet} \ar[ur,phantom, "\scalebox{1.5}{$\circlearrowleft$}"] \ar[r, "0"] \ar[d, "0"] & \overset{\C}{\bullet}  \ar[d, "1"] \\
\overset{0}{\bullet} \ar[ur, phantom, "\scalebox{1.5}{$\circlearrowleft$}"] \ar[r, "0"] & \overset{\C}{\bullet}
\end{tikzcd},\;
U^{4\phantom{0}}=
\begin{tikzcd}[]
\overset{0}{\bullet} \ar[r, "0"] \ar[d, "0"]  & \overset{0}{\bullet} \ar[d, "0"] \\
\overset{\C}{\bullet} \ar[ur,phantom, "\scalebox{1.5}{$\circlearrowleft$}"] \ar[r, "0"] \ar[d, "1"] & \overset{0}{\bullet}  \ar[d, "0"] \\
\overset{\C}{\bullet} \ar[ur, phantom, "\scalebox{1.5}{$\circlearrowleft$}"] \ar[r, "0"] & \overset{0}{\bullet}
\end{tikzcd} \\
\end{aligned}
\end{equation*}
and so forth.
It is easy to verify that these twelve rank vectors are independent (and therefore a basis of $\Z^{12}$) by inserting them as rows in a matrix and using Gauss elimination.
\end{example}

Now we can state the first main result about the $B$-isomorphism classes of representations in $R^{\iota}_{\bf d}$.

\begin{theorem}\label{thm:firstparam}
    Two representations $M^f, M^{g}$ in $R^{\iota}_{\bf d}$ are in the same orbit under the action of $G^{\iota}_{\bf d}$ given in Definition \ref{def:Baction} if and only if ${\bf r}^f={\bf r}^{g}$.
\end{theorem}
\begin{proof}
($\implies$) We suppose that $M^f$ and $M^{g}$ are $B$-isomorphic and at the same time there exist $i,j_1,j_2,k$ such that $r_{ij_1j_2k}^f \neq r_{ij_1j_2k}^{g}$. If $k\neq i+1$,  the parameters $r_{ij_1j_2k}^f$ and $r_{ij_1j_2k}^{g}$ correspond to the dimension of the intersections of the images of the compositions of the linear maps coloured in red:
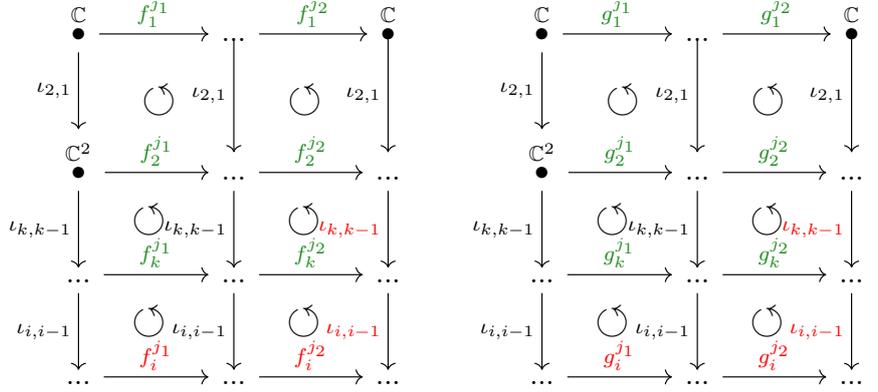
\begin{equation*}
\begin{aligned}
\begin{tikzcd}[sep=large]
\overset{\C}{\bullet} \ar[r, "\textcolor{forestgreen}{f_{1}^{j_1}}"] \ar[d,swap, "\iota_{2,1}"]  & ... \ar[r, "\textcolor{forestgreen}{f_{1}^{j_2}}"] \ar[d,swap,yshift=5, "\iota_{2,1}"] & \overset{\C}{\bullet} \ar[d,swap,,yshift=5, "\iota_{2,1}"]  \\
\overset{\C^2}{\bullet} \ar[ur,phantom, "\scalebox{1.5}{$\circlearrowleft$}"] \ar[r, "\textcolor{forestgreen}{f_{2}^{j_1}}"] \ar[d,swap, "\iota_{k,k-1}"] & ... \ar[ur, phantom,yshift=3,"\scalebox{1.5}{$\circlearrowleft$}"] \ar[r, "\textcolor{forestgreen}{f_{2}^{j_2}}"] \ar[d,swap, "\iota_{k,k-1}"] & ... \ar[d,swap, "\textcolor{red}{\iota_{k,k-1}}"] \\
... \ar[d,swap, "\iota_{i,i-1}"] \ar[ur, phantom,yshift=3, "\scalebox{1.5}{$\circlearrowleft$}"] \ar[r, "\textcolor{forestgreen}{f_{k}^{j_1}}"] & ...  \ar[ur, phantom,,yshift=3, "\scalebox{1.5}{$\circlearrowleft$}"] \ar[d, swap, "\iota_{i,i-1}"] \ar[r, "\textcolor{forestgreen}{f_{k}^{j_2}}"] & ...  \ar[d,swap, "\textcolor{red}{\iota_{i,i-1}}"]\\
... \ar[ur, phantom,yshift=3, "\scalebox{1.5}{$\circlearrowleft$}"] \ar[r, "\textcolor{red}{f_{i}^{j_1}}"] & ...  \ar[ur, phantom,yshift=3, "\scalebox{1.5}{$\circlearrowleft$}"] \ar[r, "\textcolor{red}{f_{i}^{j_2}}"] & ...
\end{tikzcd}
& \hspace{20pt}
\begin{tikzcd}[sep=large]
\overset{\C}{\bullet} \ar[r, "\textcolor{forestgreen}{g_{1}^{j_1}}"] \ar[d,swap, "\iota_{2,1}"]  & ... \ar[r, "\textcolor{forestgreen}{g_{1}^{j_2}}"] \ar[d,swap,yshift=5, "\iota_{2,1}"] & \overset{\C}{\bullet} \ar[d,swap,yshift=5, "\iota_{2,1}"]  \\
\overset{\C^2}{\bullet} \ar[ur,phantom, "\scalebox{1.5}{$\circlearrowleft$}"] \ar[r, "\textcolor{forestgreen}{g_{2}^{j_1}}"] \ar[d,swap, "\iota_{k,k-1}"] & ... \ar[ur, phantom,yshift=3,"\scalebox{1.5}{$\circlearrowleft$}"] \ar[r, "\textcolor{forestgreen}{g_{2}^{j_2}}"] \ar[d,swap, "\iota_{k,k-1}"] & ... \ar[d,swap, "\textcolor{red}{\iota_{k,k-1}}"] \\
... \ar[d,swap, "\iota_{i,i-1}"] \ar[ur, phantom,yshift=3, "\scalebox{1.5}{$\circlearrowleft$}"] \ar[r, "\textcolor{forestgreen}{g_{k}^{j_1}}"] & ...  \ar[ur, phantom,,yshift=3, "\scalebox{1.5}{$\circlearrowleft$}"] \ar[d, swap, "\iota_{i,i-1}"] \ar[r, "\textcolor{forestgreen}{g_{k}^{j_2}}"] & ...  \ar[d,swap, "\textcolor{red}{\iota_{i,i-1}}"]\\
... \ar[ur, phantom,yshift=3, "\scalebox{1.5}{$\circlearrowleft$}"] \ar[r, "\textcolor{red}{g_{i}^{j_1}}"] & ...  \ar[ur, phantom,yshift=3, "\scalebox{1.5}{$\circlearrowleft$}"] \ar[r, "\textcolor{red}{g_{i}^{j_2}}"] & ...
\end{tikzcd}.
\end{aligned}
\end{equation*}
If $k=i+1$, then $r_{ij_1j_2k}^f \neq r_{ij_1j_2k}^{g}$ means $\rk(f_{i}^{j_2}\circ\cdots\circ f_{i}^{j-1})\neq \rk (g_{i}^{j_2}\circ\cdots\circ g_{i}^{j_1})$.

The hypothesis $M^f \overset{B}{\cong} M^{g}$ means that there exists an isomorphism $\varphi$ of representations in $\rep_{\C}(Q,I)$ (represented by an an invertible, upper-triangular matrix) such that $\varphi(M^f_{i,j})=M^{g}_{i,j}$ for all $i,j$. In particular, $\varphi$ is an isomorphism when restricted to the subquivers of type $\mathbb{A}_m$:

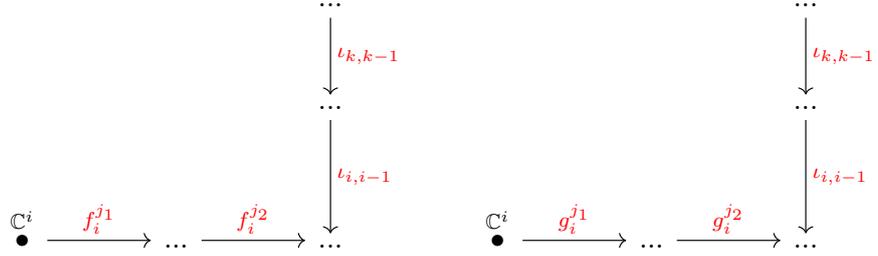
\begin{equation*}
\begin{aligned}
\begin{tikzcd}[sep=large]
 & & ... \ar[d, "\textcolor{red}{\iota_{k,k-1}}"]  \\
 & & ... \ar[d, "\textcolor{red}{\iota_{i,i-1}}"] \\
\overset{\C^i}{\bullet} \ar[r, "\textcolor{red}{f_{i}^{j_1}}"] & ... \ar[r, "\textcolor{red}{f_{i}^{j_2}}"] & ...\\
\end{tikzcd} &
\hspace{25pt}
\begin{tikzcd}[sep=large]
 & & ... \ar[d, "\textcolor{red}{\iota_{k,k-1}}"]  \\
 & & ... \ar[d, "\textcolor{red}{\iota_{i,i-1}}"] \\
\overset{\C^i}{\bullet} \ar[r, "\textcolor{red}{g_{i}^{j_1}}"] & ... \ar[r, "\textcolor{red}{g_{i}^{j_2}}"] & ...\\
\end{tikzcd}.
\end{aligned}
\end{equation*}
We apply Proposition 2.7 of \cite{abeasis1985degenerations} to these subquivers. The isomorphism $\varphi$ between $M^f$ and $M^{g}$ then implies
\begin{equation*}
\rk((f_{i}^{j_2} \circ\dots\circ f_{i}^{j_1})-(\iota_{i,i-1}\circ\dots\circ\iota_{k,k-1})=\rk((g_{i}^{j_2} \circ\dots\circ g_{i}^{j_1})-(\iota_{i,i-1}\circ\dots\circ\iota_{k,k-1}))
\end{equation*}
and
\begin{equation*}
    \rk(f_{i}^{j_2} \circ\dots\circ f_{i}^{j_1})=\rk(g_{i}^{j_2} \circ\dots\circ g_{i}^{j_1}),
\end{equation*}
contradicting the assumption $r_{ij_1j_2k}^f \neq r_{ij_1j_2k}^{g}$.

($\impliedby$) We need to show that, if ${\bf r}^{f}={\bf r}^{g}$, then $M^f\cong M^{g}$. This is equivalent to proving that to a fixed rank vector ${\bf r}^f$ corresponds exactly one decomposition into indecomposables: $M^f = \bigoplus U^{(h_1,\dots, h_n)}$. This correspondence is directly implied by Lemma \ref{lemma:indipvectors}.
\end{proof}

\begin{example}\label{ex:15rankvec}
    Theorem \ref{thm:firstparam} provides an alternative proof to the fact that the matrices representing all orbits in Example \ref{ex:15reps} define, in fact, representations that are not $B$-isomorphic. We write their rank vectors without the first six entries, since the dimension vector of all $(\Gamma,I)$-representations in $R^{\iota}_{\bf d}$ is fixed (in dimension $n+1=3$) as $(1,1,2,2,3,3)$:
    \begin{equation*}
    \begin{aligned}
        & {\bf r}^{f^{1}}=(1,1,2,1,2,3), && {\bf r}^{f^{2}}=(1,1,2,1,2,2), &&& {\bf r}^{f^{3}}=(1,1,1,1,1,2), \\
        & {\bf r}^{f^{4}}=(0,0,1,0,1,2), && {\bf r}^{f^{5}}=(1,1,1,1,2,2), &&& {\bf r}^{f^{6}}=(0,0,1,1,2,2), \\
        & {\bf r}^{f^{7}}=(0,1,1,1,2,2), && {\bf r}^{f^{8}}=(0,1,1,1,1,2), &&& {\bf r}^{f^{9}}=(1,1,1,1,1,1), \\
        & {\bf r}^{f^{10}}=(0,0,1,0,1,1), && {\bf r}^{f^{11}}=(0,0,0,0,0,1), &&& {\bf r}^{f^{12}}=(0,0,0,0,1,1), \\
        & {\bf r}^{f^{13}}=(0,0,0,1,1,1), && {\bf r}^{f^{14}}=(0,1,1,1,1,1), &&& {\bf r}^{f^{15}}=(0,0,0,0,0,0). \\
    \end{aligned}
    \end{equation*}
\end{example}

\begin{definition}
    The parametrisation of $B$-isomorphism classes of $(\Gamma,I)$-representations in $R^{\iota}_{\bf d}$ given in Definition \ref{def:rankvector} is called \Def{$r$-parametrisation}.
\end{definition}

As mentioned in Example \ref{ex:15rankvec}, once some assumptions are made on the considered representations (for instance, their dimensions vector), part of the data provided by the $r$-parametrisation is redundant. We introduce now a second, more compact parametrisation for the $(\Gamma,I)$-representations in $R^{\iota}_{\bf d}$. First, let us recall the definition of Matrix Schubert varieties and a few facts from \cite{miller2005matrix}.

\begin{definition}\cite[Definition 15.1]{miller2005matrix}\label{def:matrixschubvar}
    Let $w\in \Mat_{k\times l}$ be a partial permutation, meaning that $w$ is a $k\times l$ matrix having all entries equal to 0 except for at most one entry equal to 1 in each row and column. The \Def{matrix Schubert variety} $\overline{X}_w$ inside $\Mat_{k\times l}$ is the subvariety
    \begin{equation*}
        \overline{X}_w = \{ Z \in  \Mat_{k\times l} | \rk(Z_{p\times q})\leq \rk(w_{p\times q})\; \forall p,q\}
    \end{equation*}
    where $Z_{p\times q}$ is the upper-left $p\times q$ rectangular submatrix of $Z$.
\end{definition}

We denote by $B^-_k$ the Borel subgroup of invertible lower-triangular matrices in $\GL_k$. It follows from Definition \ref{def:matrixschubvar} that matrix Schubert varieties are preserved by the action of $B^-_k \times B_l$ on $\Mat_{k\times l}$ defined as $(b^-,b)\op Z = b^- Z b^{-1}$ (the effect of this action is ``sweeping downwards'' and ``sweeping to the right''). The following proposition implies that $B^-_k \times B_l$ has finitely many orbits in $\Mat_{k\times l}$.

\begin{proposition}\cite[Proposition 15.27]{miller2005matrix}\label{prop:northwestorbits}
    In each orbit of $B^-_k \times B_l$ on $\Mat_{k\times l}$ lies a unique partial permutation $w$, and the orbit is contained in $\overline{X}_w$.
\end{proposition}

The orbits of $B^-_k \times B_l$ on $\Mat_{k\times l}$ can then be parametrised by rank arrays: given a partial permutation $w$, $r(w)$ is the $k\times l$ array whose entry at $(p,q)$ is $\rk(w_{p\times q})$.
Since this array describes the upper-left submatrices of $w$, we call this parametrisation \Def{north-west parametrisation}. It follows from Proposition \ref{prop:northwestorbits} that two matrices $Z,Z'$ lie in the same orbit if and only if $r(Z)=r(Z')$.
Furthermore, \cite[Lemma 15.19]{miller2005matrix} and \cite[Theorem 15.31]{miller2005matrix} imply that the orbit of $w'$ under the action of $B^-_k \times B_l$ lies in the closure of the orbit of $w$ (with respect to the Zariski topology on $\Mat_{k\times l}$) if and only if $r(w')\leq r(w)$, where the "less than or equal to" relation is intended componentwise on the rank arrays.

Let us now recall the group action we considered in Definition \ref{def:Baction}. On a matrix $M$ in $\Mat_{n+1}$, this was defined as $(h_1,h_2)\op M = h_2 M h_1^{-1}$, for $(h_1,h_2)\in B_{n+1} \times B_{n+1}$. As observed in Remark \ref{remark:effectofBaction}, the effect of this action is ``sweeping upwards'' and ``sweeping to the right'', meaning that the ranks of the lower-left rectangular submatrices of $M$ are preserved.
In general, however, the group $G^{\iota}_{\bf d} = \prod_{i=1}^{n} B_{n+1}$ of Definition \ref{def:Baction} acts on a tuple of $n-1$ matrices (in ambient dimension $n+1$). Our next goal is to show that the north-west parametrisation for matrix Schubert varieties can be adapted to our context, that is, we want to describe the orbits $\orb^{\iota}_{M^f}$ in terms of ranks of certain submatrices obtained from $M^f = (f_{n+1}^1,\dots,f_{n+1}^{n-1}) \in R^{\iota}_{\bf d}$.

\begin{definition}\label{def:southwestparam}
    Given $M^f = (f_{n+1}^1,\dots,f_{n+1}^{n-1}) \in R^{\iota}_{\bf d}$, we define the \Def{south-west array} of $M^f$ as
    \begin{equation*}
        {\bf s}^f= ({\bf s}^{f_{n+1 }^{j_2}\circ \dots \circ f_{n+1 }^{j_1}})\coloneqq (\rk(f_{n+1 }^{j_2}\circ \dots \circ f_{n+1 }^{j_1})_{p\times q})
    \end{equation*}
    for $1\leq p\leq q \leq n+1$ and $1\leq j_1 \leq j_2 \leq n-1$, where $(f_{n+1 }^{j_2}\circ \dots \circ f_{n+1 }^{j_1})_{p\times q}$ is the lower-left $p\times q$ submatrix of $f_{n+1 }^{j_2}\circ \dots \circ f_{n+1 }^{j_1}$.
\end{definition}

In words, the south-west array of $M^f$ contains the lower-left ranks of all possible compositions of the linear maps that define $M^f$ in $R^{\iota}_{\bf d}$. Since all these matrices are upper-triangular, we only consider the lower-left ranks computed at entries $(p,q)$ for $p\leq q$. In order to better visualise this information, we write the components $({\bf s}^{f_{n+1 }^{j_2}\circ \dots \circ f_{n+1 }^{j_1}})$ of ${\bf s}^f$ as upper-triangular matrices as well, so that each entry contains the value of the corresponding south-west rank.

\begin{example}\label{ex:map011sw}
    We consider again the $(\Gamma,I)$-representation of Example \ref{ex:map011} and write its south-west array:
    \begin{equation*}
M^f:
\begin{tikzcd}[sep=large, ampersand replacement=\&]
\overset{\C}{\bullet} \ar[r, "\begin{bsmallmatrix}
        0
    \end{bsmallmatrix}"] \ar[d, "\iota_{2,1}"]  \& \overset{\C}{\bullet} \ar[d, "\iota_{2,1}"] \\
\overset{\C^2}{\bullet} \ar[ur,phantom, "\scalebox{1.5}{$\circlearrowleft$}"] \ar[r, "{\begin{bsmallmatrix}
        0 & 0\\
        0 & 1
    \end{bsmallmatrix}}"] \ar[d, "\iota_{3,2}"] \& \overset{\C^2}{\bullet}  \ar[d, "\iota_{3,2}"] \\
\overset{\C^3}{\bullet} \ar[ur, phantom,yshift=5, "\scalebox{1.5}{$\circlearrowleft$}"] \ar[r, "{\begin{bsmallmatrix}
        0 & 0 & 0\\
        0 & 1 & 0\\
        0 & 0 & 1
    \end{bsmallmatrix}}"] \& \overset{\C^3}{\bullet}
\end{tikzcd}.
\end{equation*}
In ambient dimension $n+1=3$, the south-west array consists of only one matrix:
\begin{equation*}
    {\bf s}^f=\Round{\begin{bsmallmatrix}
        0 & 1 & 2\\
        * & 1 & 2\\
        * & * & 1
    \end{bsmallmatrix}} .
\end{equation*}
\end{example}

\begin{theorem}\label{thm:secondparam}
    Two representations $M^f, M^{g}$ in $R^{\iota}_{\bf d}$ are in the same orbit under the action of $G^{\iota}_{\bf d}$ given in Definition \ref{def:Baction} if and only if ${\bf s}^f={\bf s}^{g}$.
\end{theorem}

\begin{proof}
    ($\implies$) If $\orb^{\iota}_{M^{g}}=\orb^{\iota}_{M^f}$, there exists an element $h\in G^{\iota}_{\bf d}$ such that $h \op M^f=M^{g}$, i.e. $h_2 f_{n+1}^1 h_1^{-1}=g_{n+1}^1, \;h_3 f_{n+1}^2 h_2^{-1}=g_{n+1}^2$ and so forth. This means that
    \begin{equation*}
        \orb^{\iota}_{M^{g_{n+1}^j}}=\orb^{\iota}_{M^{f_{n+1}^j}}
    \end{equation*}
    for all $j$, and analogously for all the possible compositions $f_{n+1 }^{j_2}\circ \dots \circ f_{n+1 }^{j_1}$ and $g_{n+1 }^{j_2}\circ \dots \circ g_{n+1 }^{j_1}$. By \cite[Proposition 15.27]{miller2005matrix}, this implies ${\bf s}^f={\bf s}^{g}$.
    
    ($\impliedby$) The claim that $M^f$ and $M^{g}$ lie in the same $G^{\iota}_{\bf d}$-orbit when ${\bf s}^f={\bf s}^{g}$ is a direct consequence of the fact that, if  ${\bf s}^f={\bf s}^{g}$, then ${\bf r}^{f}={\bf r}^{g}$  (where ${\bf r}^{f}$ is the rank vector of Definition \ref{def:rankvector}). This is true because, given a south-west array, the process of recovering the corresponding rank vector yields a unique result.
    %We illustrate the (technical) operations in Figure \ref{fig:algorithm}.
 %To simplify notation, we fix one upper-triangular matrix $g$ and show how the rank vector ${\bf r}^{g}$ is recovered from the south-west array ${\bf s}^g$.
    %The entries of ${\bf r}^{g}$ are given by $r^g_{kl}=\dim(\im(g_k)\cap \im(\iota_{k,k-1}\circ\dots\circ\iota_{l,l-1}))$, where $g_k$ denotes the restriction of $g$ to $\C^k$:
        %Then, the statement follows by applying this construction to $f_{n+1 }^{j_2}\circ \dots \circ f_{n+1 }^{j_1}$ for all $1\leq j_1\leq j_2 \leq n-1$.
\end{proof}

\begin{remark}
Thanks to Theorem \ref{thm:secondparam}, we can compute the number of $B$-orbits in dimension $n+1$ similarly to Remark \ref{remark:numberoforbits}. We sum $n-1$ times (as many as the maps $f_{n+1}^1,\dots,f_{n+1}^{n-1}$ that define $M^f$) the $n+2$-th Bell number, obtaining
\begin{equation*}
    (n-1)b_{n+2}=(n-1) \sum_{k=0}^{n+1} \binom{n+1}{k} b_{k}.
\end{equation*}
\end{remark}

\section{Linear degenerations}\label{sec:lindeg}

In this section, we briefly discuss linear degenerations of flag varieties and recall a few results about them, to then introduce our definition of linear degenerations of type A Schubert varieties. In general, to degenerate means to consider a family of varieties over $\mathbb{A}^1$, such that all fibres over $\mathbb{A}^1\setminus \{ 0\}$ are isomorphic - the general fibres - and their limit is the special fibre over $0$. 
%The term "linear", employed for instance in the study of linear degenerations of flag varieties, refers to the linear conditions that determine the variety: we vary the defining linear maps and describe how the corresponding fibres behave.
In \cite[Proposition 2.7]{cerulliirelli2012quiveranddegenerate}, the authors realise the linear degenerate flag variety as the quiver Grassmannian associated to representations of the equioriented quiver of type $\mathbb{A}_n$. This family of varieties has then been extensively studied, for instance in \cite{cerulli2017linear, fourier2020lineardegenerations}. Our main reference for this section is \cite{cerulli2017linear}, and in particular we will focus on the notion of flatness of a certain morphism of varieties.
Recall that, in Example \ref{ex: flag variety}, we described how to realise the flag variety $\Flag_{n+1}$ as a quiver Grassmannian. This is a special case arising from the following construction.

Let $V$ be a vector space of dimension $n+1$, for $n\geq 1$, and fix a basis $\mathcal{B}=\set{b_1,\dots, b_{n+1}}$. We denote by $f=(f_1,\dots,f_{n-1})$ sequences of linear maps of the form
\begin{equation*}
    \begin{tikzcd}[]
           V \arrow[r, "f_1"] & V \arrow[r, "f_2"] & \dots \arrow[r, "f_{n-2}"] & V \arrow[r, "f_{n-1}"] & V
    \end{tikzcd},
\end{equation*}
which can be seen as closed points of the variety $R=\Hom(V,V)^{n-1}$. As defined in \eqref{eq:Gaction}, the group $G=\GL(V)^n$ acts on $R$ via
\begin{equation*}
    g\op f\coloneqq (g_2 f_1 g_1^{-1}, g_3 f_2 g_2^{-1},\dots,g_n f_{n-1} g_{n-1}^{-1}).
\end{equation*}

We consider then tuples of subspaces $U=(U_1,\dots,U_n)$ in $V$ such that $\dim(U_i)=i$ for $i=1,\dots,n$. These can be seen as closed points of the product of Grassmannians
\begin{equation*}
    Z=\Gr(1,n+1)\times \Gr(2,n+1)\times \dots\times \Gr(n,n+1),
\end{equation*}
on which the group $G$ acts by translation:
\begin{equation*}
    g\op U \coloneqq (g_1U_1,g_2U_2,\dots,g_nU_n).
\end{equation*}
\begin{definition}\label{def:compatible}
    Two tuples $f\in R$ and $U\in Z$ are called \Def{compatible} if $f_i(U_i)\subseteq U_{i+1}$, for $i=1,\dots,n$.
\end{definition}
\begin{definition}\label{def:univlindeg}
    The \Def{universal linear degeneration} of the flag variety $\Flag_{n+1}$ is the variety defined as
    \begin{equation*}
        Y=\{ (f,U): f_i(U_i)\subseteq U_{i+1} \text{ for all } i=1,\dots,n \},
    \end{equation*}
    i.e. the variety of compatible pairs of sequences of maps and sequences of subspaces.
\end{definition}

Now, since $G$ acts on $R$ and $Z$, it acts componentwise on $Y$, and some properties of $Y$ can be derived from the two separate actions on $R$ and $Z$ by considering the fibres of the projections $\pi: Y\to R$ and $p:Y\to Z$. The projection $p:Y\to Z$ is $G$-equivariant, which means that it commutes with the action of $G$: $g\op p((f,U))=p(g\op(f,U))$ for all $g\in G, f\in R$ and $U\in Z$. Moreover, the space $Z$ is a homogeneous space under the $G$-action, meaning that the action is transitive, hence $Y$ is a homogeneous fibration over $Z$. If we fix a tuple $U\in Z$, we can identify its fibre via the projection $p$ with
\begin{equation*}
    \prod_{i_1}^{n-1} (\Hom(U_i,U_{i+1})\oplus \Hom(V_i,V)),
\end{equation*}
where $V_i$ is the complement of $U_i$ in $V$. These facts imply that $Y$ is a homogeneous vector bundle over $Z$, and therefore it is smooth and irreducible.
On the other hand, if we fix a tuple of linear maps $f\in R$ and consider its fibre via the projection $\pi$, we obtain the space consisting of all tuples $U\in Z$ that are compatible with $f$. In other words, each fibre $\pi^{-1}(f)$ can be viewed as a linearly degenerate version of the complete flag variety (which is the fibre over $f=(\id,\id,\dots,\id)$). 

\begin{definition}\cite[Definition 1]{cerulli2017linear}
    For a fixed $f\in R$, we call $\Flag^f_{n+1}\coloneqq \pi^{-1}(f)$ the \Def{$f$-linear degenerate flag variety}, and the map $\pi:Y\to R$ is the \Def{universal linear degeneration} of $\Flag_{n+1}$.
\end{definition}

In \cite{cerulli2017linear}, among other results, the authors use rank tuples to characterise the loci in $R$ over which $\pi$ is flat and where it is flat with irreducible fibres.
We recall that a morphism of varieties is called flat if the induced map on every stalk is a flat map of rings, and that, if a morphism is flat, then its fibres are equidimensional. For the morphism $\pi:Y\to R$, however, the authors exploit the following characterisation of flatness:
\begin{proposition}\cite[Theorem 23.1]{matsumura1989commutative}\label{prop:flatcharacterisation}
    Let $f: X\to Y$ be a morphism of varieties, where $X$ is Cohen-Macaulay and $Y$ is regular. Then, $f$ is flat if and only if its fibres are equidimensional.
\end{proposition}

We conclude this section by recalling the parametrisation of the orbits of $G$ in $R$ and reporting the result obtained in \cite{cerulli2017linear} about the flat locus of the morphism $\pi:Y \to R$.

Considering a tuple $f\in R$ is equivalent to choosing a ${\bf d}$-dimensional representation of the $\mathbb{A}_n$ equioriented quiver, for ${\bf d}=(n+1,\dots,n+1)$:
\begin{equation*}
\begin{tikzcd}[]
\overset{\C^{n+1}}{\bullet} \ar[r, "f_1"] & \overset{\C^{n+1}}{\bullet} \ar[r, "f_2"] & ... \ar[r, "f_{n-1}"] & \overset{\C^{n+1}}{\bullet},
\end{tikzcd}
\end{equation*}
and the $f$-linear degenerate flag variety is the quiver Grassmannian associated to this representation.
Then, the parametrisation of the orbits of $G$ in $R$ can be realised as a special case of the parametrisation by rank tuples given in \cite[Proposition 2.7]{abeasis1985degenerations} (see Theorem \ref{thm:paramtypeA}). In this case, the rank tuples corresponding to the orbits are of the form ${\bf r}^f=(r^f_{i,j})_{1\leq i\leq j\leq n-1}$, where $r_{i,j}=\rk(f_j\circ \dots\circ f_i)$: for $f,g$ in $R$, the orbits $\orb_f$ and $\orb_{g}$ coincide if and only if ${\bf r}^f={\bf r}^{g}$.
Additionally, as proven in \cite[Theorem 5.2]{abeasis1985degenerations}, the inclusion relations $\orb_{g}\subseteq \overline{\orb}_f$ (the Zariski closure) can be described using the same parametrisations. Such inclusion relations induce a partial ordering on the set of all $G$-orbits, and this partial ordering can be read off the rank tuples: the relation $\orb_{g}\subseteq \overline{\orb}_f$ holds if and only if ${\bf r}^{g}\leq{\bf r}^f$. In this case, we say that $\orb_f$ \Def{degenerates to} $\orb_{g}$.

\begin{example}\label{ex:extremedegflag}
    The orbit of $f=(\id,\id,\dots,\id)$ is parametrised by $r_{i,j}=n+1$ for all $i,j$, and therefore it degenerates to all other orbits. To isomorphic quiver representations correspond isomorphic quiver Grassmannians, which means that the $g$-linear degenerate flag varieties with $g\in \orb_f$ are all isomorphic to the complete flag variety $\Flag_{n+1}$. Similarly, the orbit of $f=(0,0,\dots,0)$ is parametrised by $r_{i,j}=0$ for all $i,j$, meaning that all other orbits degenerate to this one.
\end{example}

\begin{theorem}\cite[Theorem 3]{cerulli2017linear}\label{thm:flagsflatlocus}
    The flat locus of $\pi$ in $R$ is the union of all orbits degenerating to the orbit of ${\bf r}^2$, where  $r^2_{i,j}=n-j+i$ for all $i<j$.
\end{theorem}

\subsection{Linear degenerations of Shubert varieties}\label{sec:lindegschubvar}

In this section, we define linear degenerations of Schubert varieties exploiting their realisation - or their desingularisation - as quiver Grassmannians described in \cite{iezzi2025quiver}. In Section \ref{sec:qGandschubvar}, we recalled the quiver with relations $(\Gamma,I)$ and its representation $M$,
\begin{comment}
\begin{equation*}
\begin{tikzcd}[sep=large]
\overset{\C}{\bullet} \ar[r, "\id"] \ar[d, "\iota_{2,1}"]  & \overset{\C}{\bullet} \ar[r, "\id"] \ar[d, "\iota_{2,1}"] & ...\ar[r, "\id"] \ar[d, "\iota_{2,1}"] & \overset{\C}{\bullet} \ar[d, "\iota_{2,1}"] \\
\overset{\C^2}{\bullet} \ar[ur, phantom, "\scalebox{1.5}{$\circlearrowleft$}"] \ar[r, "\id"] \ar[d, "\iota_{3,2}"] & \overset{\C^2}{\bullet} \ar[ur, phantom, "\scalebox{1.5}{$\circlearrowleft$}"] \ar[r, "\id"] \ar[d, "\iota_{3,2}"]  & ... \ar[ur, phantom,xshift=5, "\scalebox{1.5}{$\circlearrowleft$}"] \ar[r, "\id"] \ar[d, "\iota_{3,2}"] & \overset{\C^2}{\bullet} \ar[d, "\iota_{3,2}"] \\
... \ar[ur, phantom,xshift=5, "\scalebox{1.5}{$\circlearrowleft$}"] \ar[r, "\id"] \ar[d, "\iota_{n+1,n}"] & ... \ar[ur, phantom,xshift=5, "\scalebox{1.5}{$\circlearrowleft$}"] \ar[r, "\id"] \ar[d, "\iota_{n+1,n}"] & ... \ar[ur, phantom, xshift=5, "\scalebox{1.5}{$\circlearrowleft$}"] \ar[r, "\id"] \ar[d,yshift=7, "\iota_{n+1,n}"] & ... \ar[d, "\iota_{n+1,n}"]\\
\overset{\C^{n+1}}{\bullet} \ar[ur, phantom, "\scalebox{1.5}{$\circlearrowleft$}"] \ar[r, "\id"] & \overset{\C^{n+1}}{\bullet} \ar[ur, phantom, "\scalebox{1.5}{$\circlearrowleft$}"] \ar[r, "\id"] & ... \ar[ur, phantom, xshift=5,yshift=5, "\scalebox{1.5}{$\circlearrowleft$}"] \ar[r, "\id"] & \overset{\C^{n+1}}{\bullet}\\
\end{tikzcd},
\end{equation*}
\end{comment}
then considered in Section \ref{sec:subvarietyR} the subvariety $R^{\iota}_{\bf d}$ inside the variety $ R_{\bf d}$ of all representations of $(\Gamma,I)$ of dimension vector ${\bf d}$. As described in Remark \ref{remark:subvariety}, the elements of this subvariety are described by tuples of linear maps $f=(f_{n+1}^1,\dots,f_{n+1}^{n-1})\in \prod_{j=1}^{n-1} U_{n+1}$, where $U$ is the subgroup of $\Mat_{n+1}$ of upper-triangular matrices with respect to the chosen basis $\mathcal{B}=\set{b_1, b_2,\dots, b_{n+1}}$ of $\C^{n+1}$. The representation corresponding to $f$ can then be visualised as follows:
\begin{equation*}
\begin{tikzcd}[sep=large]
\overset{\C}{\bullet} \ar[r, "\textcolor{forestgreen}{f_{1}^1}"] \ar[d, "\iota_{2,1}"]  & \overset{\C}{\bullet} \ar[r, "\textcolor{forestgreen}{f_{1}^2}"] \ar[d, "\iota_{2,1}"] & ...\ar[r, "\textcolor{forestgreen}{f_{1}^{n-1}}"] \ar[d, "\iota_{2,1}"] & \overset{\C}{\bullet} \ar[d, "\iota_{2,1}"] \\
\overset{\C^2}{\bullet} \ar[ur, phantom, "\scalebox{1.5}{$\circlearrowleft$}"] \ar[r, "\textcolor{forestgreen}{f_{2}^1}"] \ar[d, "\iota_{3,2}"] & \overset{\C^2}{\bullet} \ar[ur, phantom, "\scalebox{1.5}{$\circlearrowleft$}"] \ar[r, "\textcolor{forestgreen}{f_{2}^2}"] \ar[d, "\iota_{3,2}"]  & ... \ar[ur, phantom,xshift=5, "\scalebox{1.5}{$\circlearrowleft$}"] \ar[r, "\textcolor{forestgreen}{f_{2}^{n-1}}"] \ar[d, "\iota_{3,2}"] & \overset{\C^2}{\bullet} \ar[d, "\iota_{3,2}"] \\
... \ar[ur, phantom,xshift=5, "\scalebox{1.5}{$\circlearrowleft$}"] \ar[r, "\textcolor{forestgreen}{f_{i}^1}"] \ar[d, "\iota_{n+1,n}"] & ... \ar[ur, phantom,xshift=5, "\scalebox{1.5}{$\circlearrowleft$}"] \ar[r, "\textcolor{forestgreen}{f_{i}^j}"] \ar[d, "\iota_{n+1,n}"] & ... \ar[ur, phantom, xshift=5, "\scalebox{1.5}{$\circlearrowleft$}"] \ar[r, "\textcolor{forestgreen}{f_{i}^{n-1}}"] \ar[d,yshift=5, "\iota_{n+1,n}"] & ... \ar[d, "\iota_{n+1,n}"]\\
\overset{\C^{n+1}}{\bullet} \ar[ur, phantom, "\scalebox{1.5}{$\circlearrowleft$}"] \ar[r, "\textcolor{forestgreen}{f_{n+1}^1}"] & \overset{\C^{n+1}}{\bullet} \ar[ur, phantom, "\scalebox{1.5}{$\circlearrowleft$}"] \ar[r, "\textcolor{forestgreen}{f_{n+1}^2}"] & ... \ar[ur, phantom, xshift=5,yshift=5, "\scalebox{1.5}{$\circlearrowleft$}"] \ar[r, "\textcolor{forestgreen}{f_{n+1}^{n-1}}"] & \overset{\C^{n+1}}{\bullet}\\
\end{tikzcd},
\end{equation*}

\vspace{-1cm}

where each map $f_{i}^j$ is represented by the appropriate submatrix of $f_{n+1}^j$ (i.e., it is the restriction of $f_{n+1}^j$ to $\C^i$).

In \cite{iezzi2025quiver}, we gave two different constructions for two dimension vectors for the quiver $(\Gamma,I)$.
The dimension vector ${\bf r}^w$ was defined as
\begin{equation*}
  r^w_{i,j}\coloneqq \#\set{k\leq j : w(k)\leq i}, \quad i=1,\dots,n+1,\quad j=1,\dots,n
\end{equation*}
and allowed us to recover the Bott-Samelson resolution of the Schubert variety $X_w$ via the quiver Grassmannian $\Gr_{{\bf r}^w}(M)$.
The second dimension vector, denoted by ${\bf e}^w$ and employed only when $X_w$ is a smooth variety, was defined by
\begin{equation*}
\begin{cases}
    e^w_{i,j}\coloneqq r^w_{i,j} & \text{ if } r^w_{i,j} =\min \{ i,j \} \\
    & \text{ or } r^w_{i,j} =0\\
    e^w_{i,j}\coloneqq \max \{ e^w_{i-1,j}, e^w_{i,j-1}  \} & \text{ if } 0 < r^w_{i,j} <\min \{ i,j \}
\end{cases}
\end{equation*}
and used to find an explicit isomorphism between the quiver Grassmannian $\Gr_{{\bf e}^w}(M)$ and the considered Schubert variety.
The goal of this section is to introduce linear degenerations of Schubert varieties; for this purpose, we will define a universal linear degeneration of (the considered Schubert variety) $X_w$, whose specific construction will depend on the fixed permutation $w$.
For permutations $w$ such that $X_w$ is a singular variety, we will employ the corresponding dimension vector ${\bf r}^w$. For permutations $w$ that yield a smooth Schubert variety $X_w$, instead, the dimension vector to be considered to construct linear degenerations of $X_w$ is ${\bf e}^w$. However, in order to simplify notation throughout the construction, we will simply denote by ${\bf e}^w$ the dimension vector that corresponds to the fixed permutation $w$, without subdividing the notation into cases.

Let us now fix a permutation $w$ in $S_{n+1}$ and denote by $Z$ the product of Grassmannians
\begin{equation*}
    Z=\prod_{i,j}\Gr(e^w_{i,j},\C^i),
\end{equation*}
for $i=1,\dots,n+1$ and $j=1,\dots,n$, where ${\bf e}^w$ is the dimension vector constructed from $w$, and by $U=(U_{ij})$ a closed point in $Z$. We consider the action of the group $G^{\iota}_{\bf d}\coloneqq \prod_{i=1}^{n} B_{n+1}$ on $R^{\iota}_{\bf d}$ given in Definition \ref{def:Baction}:
\begin{equation*}
    h \op f= (h_2 f_{n+1}^1 h_1^{-1}, h_3 f_{n+1}^2 h_2^{-1}, \dots, h_n f_{n+1}^{n-1} h_{n-1}^{-1})
\end{equation*}
for some $h \in G^{\iota}_{\bf d}$. This action is then extended to the other linear maps $f_{i}^j$, for $i<n+1$, as explained in Remark \ref{remark:actionrestriction}: each $f_{i}^j$ is acted upon by the restrictions of the maps $(h_1,\dots,h_n)$ to $\C^i$.
The group $G^{\iota}_{\bf d}$ acts on $Z$ by translation:
\begin{equation*}
    h\op U\coloneqq (h_1 U_{n+11}, h_2 U_{n+12},\dots,h_n U_{n+1n})
\end{equation*}
and this action is extended to the other subspaces $U_{ij}$, for $i<n+1$, by letting the restrictions of the maps $(h_1,\dots,h_n)$ to $\C^i$ act on $U_{ij}$ by translation.

\begin{definition}
    We call a pair $(f,U)$ \Def{compatible} if $f_{\alpha}(U_{s(\alpha)})\subseteq U_{t(\alpha)}$ for all arrows $\alpha$ in $(\Gamma,I)_1$.
\end{definition}

 %We would like to remark that, even though at first glance it might seem unnecessary to describe this trivial extension of the $G^{\iota}_{\bf d}$-action, considering the whole quiver $(\Gamma,I)$ instead of only its last row is fundamental from a geometrical point of view: whether $U$ is compatible with some fixed $f$ depends on all subspaces $U_{i,j}$, which are realised in (potentially) any row of $(\Gamma,I)$, according to the entries of ${\bf e}^w$.

\begin{definition}
    The \Def{universal linear degeneration} of the Schubert variety $X_w$ is the variety defined as
     \begin{equation*}
        Y=\{ (f,U): f_{\alpha}(U_{s(\alpha)})\subseteq U_{t(\alpha)} \text{ for all } \alpha \in  (\Gamma,I)_1\}.
    \end{equation*}
\end{definition}

We can then consider the two projections $\pi: Y\to R^{\iota}_{\bf d}$ and $p: Y\to Z$.
The action of $G^{\iota}_{\bf d}$ on $R^{\iota}_{\bf d}$ and $Z$ induces the action of $G^{\iota}_{\bf d}$ on $Y$, and the projection $p$ is $G^{\iota}_{\bf d}$-equivariant, but in this case the action of  $G^{\iota}_{\bf d}$ on $Z$ is not transitive. %This means that we cannot employ the same tools as for the universal linear degeneration of the flag variety in order to deduce smoothness or irreducibility of $Y$.
Now, if we fix a tuple of maps $f$ in $R^{\iota}_{\bf d}$ and consider its fibre via the projection $\pi$, we obtain the space consisting of all $U\in Z$ that are compatible with $f$. This means that each fibre $\pi^{-1}(f)$ can be viewed as a linearly degenerate version of the fixed Schubert variety $X_w$, which is itself the fibre over $f=(\id,\id,\dots,\id)$.

\begin{definition}\label{def:lindegschubvar}
    For a fixed $f\in R^{\iota}_{\bf d}$, we call $X^f_w\coloneqq \pi^{-1}(f)$ the \Def{$f$-linear degenerate Schubert variety}, and the map $\pi:Y\to  R^{\iota}_{\bf d}$ is the \Def{universal linear degeneration} of $X_w
    $.
\end{definition}

In other words, the $f$-linear degenerate Schubert variety is defined as the quiver Grassmannian $\Gr_{{\bf e}^w}(M^f)$, where $M^f$ is the $(\Gamma,I)$-representation associated to $f$ in the sense of Section \ref{sec:subvarietyR}. By comparing the definition of linear degenerations of flag varieties and our definition of linear degenerations of Schubert varieties, the restriction we operated in Section \ref{sec:subvarietyR} on the considered $(\Gamma,I)$-representations appears now more reasonable. We kept the vertical maps fixed as standard inclusions and varied the horizontal maps, meaning that we degenerate the conditions defining each flag in $\Flag_{n+1}$ but not the combinatorial conditions that determine which flags belong to $X_w$.

In Section \ref{sec:parametrisations}, we provided two different parametrisations for the  $G^{\iota}_{\bf d}$-orbits in $R^{\iota}_{\bf d}$: the rank vectors, in Definition \ref{def:rankvector}, and the south-west arrays, in Definition \ref{def:southwestparam}.
Furthermore, the south-west parametrisation allows us to describe the inclusion relations of the form $\orb^{\iota}_{M^{g}} \subseteq \overline{\orb}^{\iota}_{M^f}$: we derive the following corollary from  \cite[Lemma 15.19]{miller2005matrix} and \cite[Theorem 15.31]{miller2005matrix}.

\begin{corollary}\label{cor:orbitdegen}
     The orbit $\orb^{\iota}_{M^{g}}$ of $M^{g}$ under the action of $G^{\iota}_{\bf d}$ lies in the closure of the orbit $\orb^{\iota}_{M^f}$ of $M^f$ (with respect to the Zariski topology on $\Mat_{n+1}$) if and only if ${\bf s}^f\leq {\bf s}^{g}$, where the "less than or equal to" relation is intended componentwise on the south-west arrays. In this case, we write $\orb^{\iota}_{M^{g}} \subseteq \overline{\orb}^{\iota}_{M^f}$ and we say that $\orb^{\iota}_{M^{f}}$ degenerates to $\orb^{\iota}_{M^{g}}$. 
\end{corollary}

\begin{example}\label{ex:extremedeg}
    Analogously to Example \ref{ex:extremedegflag}, the orbit of $\id=(\id,\id,\dots,\id)$ is described by the largest possible south-west ranks, and therefore it degenerates to all other orbits. Since isomorphic quiver representations yield isomorphic quiver Grassmannians, the $g$-linear degenerate Schubert varieties with $g\in \orb_f$ are all isomorphic to the Schubert variety $X_w$. The ``most degenerate'' orbit is that of $0=(0,0,\dots,0)$: it is parametrised by ${\bf s}_{i,j}=0$ for all $i,j$, hence all other orbits degenerate to this one. In this case, the linear maps corresponding to the horizontal arrows of $(\Gamma,I)$ are not imposing any conditions between the subspaces $U_{i,j}$ that belong to different columns of $(\Gamma,I)$. This means that the 0-linear degenerate Schubert variety is given by the product of $n$ partial flag varieties (one for each column of $(\Gamma,I)$), where the dimensions of the free subspaces are given by the corresponding entries of ${\bf e}^w$.
\end{example}

\begin{example}\label{ex:ordering}
    Consider again the $(\Gamma,I)$-representation of Example \ref{ex:map011sw}:
       \begin{equation*}
f:
\begin{tikzcd}[sep=large, ampersand replacement=\&]
\overset{\C}{\bullet} \ar[r, "\begin{bsmallmatrix}
        0
    \end{bsmallmatrix}"] \ar[d, "\iota_{2,1}"]  \& \overset{\C}{\bullet} \ar[d, "\iota_{2,1}"] \\
\overset{\C^2}{\bullet} \ar[ur,phantom, "\scalebox{1.5}{$\circlearrowleft$}"] \ar[r, "{\begin{bsmallmatrix}
        0 & 0\\
        0 & 1
    \end{bsmallmatrix}}"] \ar[d, "\iota_{3,2}"] \& \overset{\C^2}{\bullet}  \ar[d, "\iota_{3,2}"] \\
\overset{\C^3}{\bullet} \ar[ur, phantom,yshift=5, "\scalebox{1.5}{$\circlearrowleft$}"] \ar[r, "{\begin{bsmallmatrix}
        0 & 0 & 0\\
        0 & 1 & 0\\
        0 & 0 & 1
    \end{bsmallmatrix}}"] \& \overset{\C^3}{\bullet}
\end{tikzcd},
\end{equation*}
whose rank vector and south-west array are, respectively:
\begin{equation*}
{\bf r}^f=(0,0,1,0,1,2), \quad
    {\bf s}^f=\Round{\begin{bsmallmatrix}
        0 & 1 & 2\\
        * & 1 & 2\\
        * & * & 1
    \end{bsmallmatrix}}.
\end{equation*}
The $(\Gamma,I)$-representation
    \begin{equation*}
g:
\begin{tikzcd}[sep=large, ampersand replacement=\&]
\overset{\C}{\bullet} \ar[r, "\begin{bsmallmatrix}
        0
    \end{bsmallmatrix}"] \ar[d, "\iota_{2,1}"]  \& \overset{\C}{\bullet} \ar[d, "\iota_{2,1}"] \\
\overset{\C^2}{\bullet} \ar[ur,phantom, "\scalebox{1.5}{$\circlearrowleft$}"] \ar[r, "{\begin{bsmallmatrix}
        0 & 1\\
        0 & 0
    \end{bsmallmatrix}}"] \ar[d, "\iota_{3,2}"] \& \overset{\C^2}{\bullet}  \ar[d, "\iota_{3,2}"] \\
\overset{\C^3}{\bullet} \ar[ur, phantom,yshift=5, "\scalebox{1.5}{$\circlearrowleft$}"] \ar[r, "{\begin{bsmallmatrix}
        0 & 1 & 0\\
        0 & 0 & 0\\
        0 & 0 & 0
    \end{bsmallmatrix}}"] \& \overset{\C^3}{\bullet}
\end{tikzcd},
\end{equation*}
is evidently not in the same orbit as $f$, because it has rank equal to one instead of two (the rank of a matrix is the same as its south-west rank $s_{1,n+1}$).
Its rank vector and south-west array are:
\begin{equation*}
{\bf r}^{g}=(0,1,1,1,1,1), \quad
    {\bf s}^{g}=\Round{\begin{bsmallmatrix}
        0 & 1 & 1\\
        * & 0 & 0\\
        * & * & 0
    \end{bsmallmatrix}}.
\end{equation*}
The relation between the two south-west arrays is ${\bf s}^{g}\leq {\bf s}^f$, which implies $\orb^{\iota}_{M^{g}} \subseteq \overline{\orb}^{\iota}_{M^f}$. Notice that comparing their rank vectors would not yield the same result: in fact, ${\bf r}^{g}$ and ${\bf r}^f$ are not comparable. An example of a $(\Gamma,I)$-representation that is neither above nor below $f$ in the poset of the $G^{\iota}_{\bf d}$-orbits is
 \begin{equation*}
h:
\begin{tikzcd}[sep=large, ampersand replacement=\&]
\overset{\C}{\bullet} \ar[r, "\begin{bsmallmatrix}
        1
    \end{bsmallmatrix}"] \ar[d, "\iota_{2,1}"]  \& \overset{\C}{\bullet} \ar[d, "\iota_{2,1}"] \\
\overset{\C^2}{\bullet} \ar[ur,phantom, "\scalebox{1.5}{$\circlearrowleft$}"] \ar[r, "{\begin{bsmallmatrix}
        1 & 0\\
        0 & 0
    \end{bsmallmatrix}}"] \ar[d, "\iota_{3,2}"] \& \overset{\C^2}{\bullet}  \ar[d, "\iota_{3,2}"] \\
\overset{\C^3}{\bullet} \ar[ur, phantom,yshift=5, "\scalebox{1.5}{$\circlearrowleft$}"] \ar[r, "{\begin{bsmallmatrix}
        1 & 0 & 0\\
        0 & 0 & 0\\
        0 & 0 & 0
    \end{bsmallmatrix}}"] \& \overset{\C^3}{\bullet}
\end{tikzcd}.
\end{equation*}
Its south-west array is ${\bf s}^{h}=\Round{\begin{bsmallmatrix}
        1 & 1 & 1\\
        * & 0 & 0\\
        * & * & 0
    \end{bsmallmatrix}}$, which is not comparable to ${\bf s}^{f}$. This example highlights a difference between the south-west parametrisation and the parametrisation of the $G$-orbits, for $G=\GL(V)^n$, and consequently between their posets. In the case of the $G$-action, it is enough to compare the ranks of (all compositions of) the matrices, while, for comparing $G^{\iota}_{\bf d}$-orbits, we need to check which entries of the matrices actually contribute to the rank.
\end{example}

\begin{example}
In Example \ref{ex:15reps}, we listed the standard representatives of the fifteen $B$-orbits in dimension $n+1=3$. The relations between (the closures of) such $B$-orbits are represented by the following poset:
    
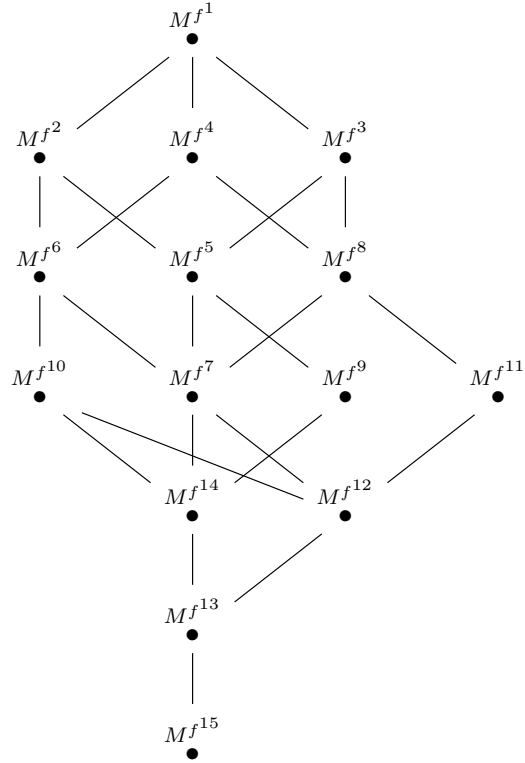
\begin{equation*}
        \begin{tikzcd}[]
        & \overset{M^{f^1}}{\bullet} \\
            \overset{M^{f^2}}{\bullet} \ar[ur, dash] & \overset{M^{f^4}}{\bullet} \ar[u, dash] & \overset{M^{f^3}}{\bullet} \ar[ul, dash] \\
            \overset{M^{f^6}}{\bullet} \ar[u, dash] \ar[ur, dash] & \overset{M^{f^5}}{\bullet} \ar[ul, dash] \ar[ur, dash] & \overset{M^{f^8}}{\bullet} \ar[u, dash] \ar[ul, dash]\\
            \overset{M^{f^{10}}}{\bullet} \ar[u, dash] & \overset{M^{f^{7}}}{\bullet} \ar[ul, dash] \ar[u, dash] \ar[ur, dash] & \overset{M^{f^{9}}}{\bullet} \ar[ul, dash] & \overset{M^{f^{11}}}{\bullet} \ar[ul, dash] \\
            & \overset{M^{f^{14}}}{\bullet} \ar[u, dash] \ar[ul, dash] \ar[ur, dash] & \overset{M^{f^{12}}}{\bullet} \ar[ull, dash] \ar[ul, dash] \ar[ur, dash] \\
            & \overset{M^{f^{13}}}{\bullet} \ar[u, dash] \ar[ur, dash] \\
            & \overset{M^{f^{15}}}{\bullet} \ar[u, dash]
        \end{tikzcd}
    \end{equation*}
\end{example}

\begin{remark}
    The ordering induced by the inclusion relations between $G^{\iota}_{\bf d}$-orbits is a refinement of the  ordering induced by the inclusion relations between $G$-orbits, in the sense that elements that are equal under the second ordering might not be equal under the first one. This means that the ordering induced by the $G^{\iota}_{\bf d}$-action is stronger than that induced by the $G$-action: if $a\leq b$ holds under the first ordering, then it holds under the second one.
\end{remark}

\begin{remark}
    In some special cases, computing the dimension of the quiver Grassmannian $\Gr_{{\bf e}^w}(M^f)$ (that is, of the $f$-linear degenerate Schubert variety) is particularly easy. If $f=0=(0,\dots,0)$, as explained in Example \ref{ex:extremedeg}, then $\Gr_{{\bf e}^w}(M^f)$ is a product of partial flag varieties: its dimension is  the sum of the dimensions of each partial flag variety. In general, however, the resulting variety is much more complicated, and so is computing its dimension. We only have a bound on the dimensions of its irreducible components: for any $f\in R^{\iota}_{\bf d}$ and ${\bf d}=\dimvec M^{\id}$, all irreducible components of $\Gr_{{\bf e}^w}(M^f)$ have dimension at least $\langle {\bf e}^w, {\bf d}-{\bf e}^w \rangle$, which is the dimension of $\Gr_{{\bf e}^w}(M^{\id})$ by \cite[Corollary 4.6]{iezzi2025quiver}.
\end{remark}

Before making a few more examples of linear degenerate Schubert varieties and opening the discussion about the flat locus of the morphism $\pi:Y\to  R^{\iota}_{\bf d}$, we would like to complete the description of the south-west parametrisation. We know that a tuple $f\in R^{\iota}_{\bf d}$ is parametrised by its south-west array, but which tuples of non-negative integers are the south-west array of some tuple $f\in R^{\iota}_{\bf d}$?
The simple answer to this question follows directly from the definition of south-west arrays: if we consider a south-west array ${\bf s}$ as a matrix in $U\subseteq \Mat_{n+1}$ (an upper-triangular matrix), then this matrix is parametrised by itself. In other words, a matrix $A$ whose entry $a_{i,j}$ is a south-west rank has precisely those $a_{i,j}$ as its south-west ranks.
This means that, in order to determine whether a tuple of non-negative integers is a south-west array, we can check if it parametrises itself. Nonetheless, we can write down the conditions that determine if a tuple of non-negative integers ${\bf \omega}$ is a south-west array of some $f\in R^{\iota}_{\bf d}$, in ambient dimension $n+1$.
Let us denote the entries of ${\bf \omega}$ by $w^{i,j}_{a,b}$; imposing the following conditions implies that $w^{i,j}_{a,b}$ is the south-west rank computed at entry $(i,j)$ of the matrix given by $f_b\circ\dots\circ f_a$, for $1\leq a \leq b \leq n-1$ and $1\leq i \leq j\leq n+1$:
\begin{equation}\label{eq:sufficientcond}
    \begin{cases}
        w^{i,j}_{a,b} \leq \min \{ j-i+1,j\}\\
        w^{i,j}_{a,b}\leq \min \{ w^{i,n}_{a,a+t}, w^{1,j}_{a+t,b}\}\\
        p_{i,j}\leq w^{i,j}_{a,b} \leq p_{i,j} +1
    \end{cases}
\end{equation}
for all $a<t<b$, where $p_{i,j}$ is the number of pivots in the south-west $i\times j$ submatrix and $ w^{i,j}_{a,b}= p_{i,j} +1$ is only allowed if there are no pivots in row $i$ or column $j$.
The first inequality is necessary for $w^{i,j}_{a,b}$ to be the rank of the (upper-triangular) $i\times j$ submatrix, while the second inequality represents the fact that the rank of a product ($w^{i,j}_{a,b}$) cannot be bigger than each of the ranks of the two factors; since $w^{i,j}_{a,b}$ is the result of the composition $f_b\circ\dots\circ f_a$, this has to hold independently of how we apply associativity.
The last condition ensures that $w^{i,j}_{a,b}$ counts the number of pivots in the $i \times j$ submatrix and that it does not increase by more than one in each row and column.
We omit the analogous discussion for rank vectors, since it is of technical nature and falls outside the main concerns of this paper.

\section{On the flat locus of linear degenerations of Schubert varieties}\label{sec:flatlocus}

In this section, we open the question of the flat locus of the morphism $\pi:Y\to  R^{\iota}_{\bf d}$. We would like to exploit the characterisation of the flat locus given in Proposition \ref{prop:flatcharacterisation}; we know that $R^{\iota}_{\bf d}$ is regular (it is a smooth subgroup of the group $\GL_{n+1}^n$), but we need to show that $Y$ is a Cohen-Macaulay variety.
A variety $X$ is called Cohen-Macaulay if for every point $x\in X$ there exists an affine open neighbourhood $U\subset X$ of $x$ such that the ring of regular functions $\orb_X(U)$ is Noetherian and Cohen-Macaulay.
More information on Cohen-Macaulay rings can be found, for instance, in \cite{bruns1998cohen}. However, in order to prove that $Y$ is Cohen-Macaulay, we can apply the following:

\begin{lemma}\cite[Lemma 10.135.3]{stack10.135}\label{lemma:completeint}
    Let $\K$ be a field. Let $S$ be a finite type $\K$-algebra. If $S$ is locally a complete intersection, then $S$ is a Cohen-Macaulay ring.
\end{lemma}

In particular, we would like to make use of an analogous strategy to the one employed in the proof of Theorem 11 in \cite{cerulli2017linear}: here, the authors show that the $f$-linear degenerate flag variety (for $f$ such that ${\bf r}^f={\bf r}^2$, as in Theorem \ref{thm:flagsflatlocus}) is locally a complete intersection. In order to do so, we realise our quiver Grassmannians as geometric quotients of the appropriate varieties; we follow \cite[Section 2.3]{cerulliirelli2012quiveranddegenerate}.

Let ${\bf e}^w$ be as above and $f\in  R^{\iota}_{\bf d}$ such that $\dim(\Gr_{{\bf e}^w}(M^f))=\dim(X_w)$ (this is a necessary condition for flatness). We define the vector space
\begin{equation*}
    V\coloneqq R_{{\bf e}^w}\times \prod_{(p,q)\in (\Gamma,I)_0} \Hom(\C^{e^w_{p,q}}, \C^i),
\end{equation*}
where $R_{{\bf e}^w}$ is the variety of representations of $(\Gamma,I)$ with dimension vector ${\bf e}^w$. We denote the elements of $V$ by $((N_{\alpha}),(g_{p,q}))$, where $\alpha$ varies over all arrows in $(\Gamma,I)_1$ and $(p,q)$ corresponds to a vertex of $(\Gamma,I)$.
Then, we consider the affine variety $\Hom({\bf e}^w, M^f)$ in $V$ consisting of tuples $((N_{\alpha}),(g_{p,q}))$ that satisfy
\begin{equation}\label{eq:homvariety}
    \begin{cases}
        f_{i}^j\circ g_{i,j}=g_{i,j+1}\circ N_{\alpha}\; \text{ for }\alpha :s(\alpha)=(i,j), t(\alpha)=(i,j+1)\\
        \iota_{i+1,i}\circ g_{i,j}=g_{i+1,j}\circ N_{\alpha}\; \text{ for }\alpha : s(\alpha)=(i,j), t(\alpha)=(i+1,j)
    \end{cases}.
\end{equation}
 
In words, we consider all possible tuples $((N_{\alpha}),(g_{p,q}))$ where $N$ is a $(\Gamma,I)$-representation of dimension vector ${\bf e}^w$ and $g=(g_{p,q})$, for $(p,q)\in (\Gamma,I)_0$, is a morphism of representations from $N$ to $M^f$. The relations in \eqref{eq:homvariety} impose all commutativity relations that are necessary for $g$ to be a morphism of representations. Figure \ref{fig:3Ddiagram} illustrates the diagram in ambient dimension $n+1=3$: the representations $N$ and $M^f$ satisfy the relations given by $(\Gamma,I)$, while the commutativity of all other squares in the diagram, which involve the morphism $g$, is imposed by the relations in \eqref{eq:homvariety}.

\begin{figure}[h]
    \centering  
\begin{tikzcd}[sep=small]
 & & \C^{e^w_{1,1}} \arrow[dl, swap, "N_{\alpha_2}"] \arrow[rrr, "N_{\alpha_1}"] \arrow[dd, shorten >= -3, "g_{1,1}"] & & & \C^{e^w_{1,2}} \arrow[dl, swap, "N_{\alpha_4}"] \arrow[dd, "g_{1,2}"] \\
 & \C^{e^w_{2,1}} \arrow[dl,swap, "N_{\alpha_5}"] \arrow[rrr, crossing over, "N_{\alpha_3}"] \arrow[dd, yshift=-2,"g_{2,1}"] & & & \C^{e^w_{2,2}} \arrow[dl, swap, "N_{\alpha_7}"] \\
|[yshift=17]| \C^{e^w_{3,1}} \arrow[rrr,crossing over, "N_{\alpha_6}"] \arrow[dd, "g_{3,1}"] & & \C \arrow[dl,swap, "\iota_{2,1}"] \arrow[rrr, "f_1^1"] & |[yshift=17]| \C^{e^w_{3,2}} & & \C \arrow[dl, "\iota_{2,1}"] \\
& \C^2 \arrow[rrr, "f_2^1"] \arrow[dl, swap, "\iota_{3,2}"] & & & \C^2 \arrow[from=uu, crossing over, "g_{2,2}"] \arrow[dl, "\iota_{3,2}"]\\
\C^3 \arrow[rrr, "f_3^1"] & & & \C^3 \arrow[from=uu, crossing over, "g_{3,2}"]
 \end{tikzcd}
 \caption{The diagram formed by $M^f$, $N$ and $g$}
\label{fig:3Ddiagram}
\end{figure}
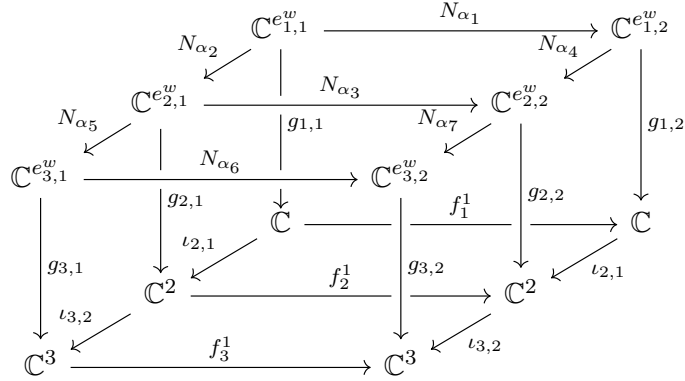

Then, the quiver Grassmannian $\Gr_{{\bf e}^w}(M^f)$ can be realised as the following geometric quotient:
\begin{equation*}
    \Gr_{{\bf e}^w}(M^f)\cong \Hom^0({\bf e}^w, M^f)/G_{{\bf e}^w},
\end{equation*}

where $G_{{\bf e}^w}=\prod_{(i,j)\in (\Gamma,I)_0} \GL(\C^{e^w_{i,j}})$, i.e. the product of the general linear groups acting naturally on the vector spaces of $N\in  R_{{\bf e}^w}$, and $\Hom^0({\bf e}^w, M^f)$ is the open subvariety in $\Hom({\bf e}^w, M^f)$ defined by $g_{p,q}$ being an injective map, for all $(p,q)\in (\Gamma,I)_0$. Since $\Hom^0({\bf e}^w, M^f)$ and $\Hom({\bf e}^w, M^f)$ have the same codimension in $V$ (because the first one is an open subvariety of the second one), it is enough to show that $\Hom^0({\bf e}^w, M^f)$ is locally a complete intersection.

In order to do so, we want to make use of a dimension formula arising from the quiver Grassmannian being a geometric quotient, namely that
\begin{equation}
    \dim(\Gr_{{\bf e}^w}(M^f))=\dim(\Hom^0({\bf e}^w, M^f))-\dim(G_{{\bf e}^w}).
\end{equation}

This would allow us to deduce the codimension of $\Hom^0({\bf e}^w, M^f)$ in $V$ and to compare it with the number of equations defining $\Hom({\bf e}^w, M^f)$ in $V$: if we obtain the same number, then $\Gr_{{\bf e}^w}(M^f)$ is locally a complete intersection. In this case, by Lemma \ref{lemma:completeint}, we would be able to describe completely the flat locus of $\pi : Y\to R^{\iota}_{\bf d}$: it would consist precisely of all $f \in R^{\iota}_{\bf d}$ such that $\dim(\Gr_{{\bf e}^w}(M^f))=\dim(X_w)$.

\begin{conjecture}\label{conj:flatlocus}
    A tuple $f$ is in the flat locus of $\pi : Y\to R^{\iota}_{\bf d}$ if and only if $\dim(\Gr_{{\bf e}^w}(M^f))=\dim(X_w)$.
\end{conjecture}

We present now an example in ambient dimension $n+1=3$ motivating our conjecture.

Let us fix the permutation $w=[231]\in S_3$ and the representation determined by the identity map $f$:
\begin{equation*}
f:
\begin{tikzcd}[sep=large, ampersand replacement=\&]
\overset{\C}{\bullet} \ar[r, "\begin{bsmallmatrix}
        1
    \end{bsmallmatrix}"] \ar[d, "\iota_{2,1}"]  \& \overset{\C}{\bullet} \ar[d, "\iota_{2,1}"] \\
\overset{\C^2}{\bullet} \ar[ur,phantom, "\scalebox{1.5}{$\circlearrowleft$}"] \ar[r, "{\begin{bsmallmatrix}
        1 & 0\\
        0 & 1
    \end{bsmallmatrix}}"] \ar[d, "\iota_{3,2}"] \& \overset{\C^2}{\bullet}  \ar[d, "\iota_{3,2}"] \\
\overset{\C^3}{\bullet} \ar[ur, phantom,yshift=5, "\scalebox{1.5}{$\circlearrowleft$}"] \ar[r, "{\begin{bsmallmatrix}
        1 & 0 & 0\\
        0 & 1 & 0\\
        0 & 0 & 1
    \end{bsmallmatrix}}"] \& \overset{\C^3}{\bullet}
\end{tikzcd}.
\end{equation*}
The dimension vector ${\bf e}^w$ is then
\begin{equation*}
    {\bf e}^w=\begin{bsmallmatrix}
        0 & 0\\
        1 & 1\\
        1 & 2
    \end{bsmallmatrix},
\end{equation*}
and the quiver Grassmannian $\Gr_{{\bf e}^w}(M^f))$ is isomorphic to the Schubert variety $X_{[231]}$: it is a smooth projective variety of dimension $\ell(w)=2$.
We choose this representation specifically because we know that it is in the flat locus of $\pi : Y\to R^{\iota}_{\bf d}$ (its fibre is the nondegenerate Schubert variety).

\begin{remark}
    If we choose a permutation in $S_3$ with length strictly smaller than $2$, then $\Gr_{{\bf e}^w}(M^f))=\Gr_{{\bf e}^w}(M))$ independently of the representation $f$. This can be seen by computing the corresponding dimension vectors
    \begin{equation*}
        {\bf e}^{\id}=\begin{bsmallmatrix}
        1 & 1\\
        1 & 2\\
        1 & 2
    \end{bsmallmatrix},\;
    {\bf e}^{[213]}=\begin{bsmallmatrix}
        0 & 1\\
        1 & 2\\
        1 & 2
    \end{bsmallmatrix},\;
    {\bf e}^{[132]}=\begin{bsmallmatrix}
        1 & 1\\
        1 & 1\\
        1 & 2
    \end{bsmallmatrix}
    \end{equation*}
    and noticing that all free subspaces are realised in the first column of $(\Gamma,I)$, meaning that the linear maps between the first and second column do not affect the quiver Grassmannian.
\end{remark}

We consider now the geometric quotient
\begin{equation*}
    \Gr_{{\bf e}^w}(M^f)\cong \Hom^0({\bf e}^w, M^f)/G_{{\bf e}^w},
\end{equation*}

and want to show that $\Hom^0({\bf e}^w, M^f)$ is locally a complete intersection. Figure \ref{fig:3Ddiagramexample} represents $M^f$ and the diagram it forms together with a representation in $R_{{\bf e}^w}$ and a morphism $g$.

\begin{figure}[h]
    \centering  
\begin{tikzcd}[ampersand replacement= \&]
 \& \& 0 \arrow[dl, swap, "{[0]}"] \arrow[rrr, "{[0]}"] \arrow[dd, shorten >= -3, "{[0]}"] \& \& \& 0 \arrow[dl, swap, "{[0]}"] \arrow[dd, "{[0]}"] \\
 \& \C \arrow[dl,swap, "N_2"] \arrow[rrr, crossing over, "N_1"] \arrow[dd, yshift=-4, swap, "g_{2,1}"] \& \& \& \C \arrow[dl, swap, "N_4"] \\
|[yshift=17]| \C \arrow[rrr,crossing over, "N_3"] \arrow[dd, "g_{3,1}"] \& \& \C \arrow[dl,swap, "\iota_{2,1}"] \arrow[rrr, "f_1^1"] \& |[yshift=17]| \C^2 \& \& \C \arrow[dl, "\iota_{2,1}"] \\
\& \C^2 \arrow[rrr, "f_2^1"] \arrow[dl, swap, "\iota_{3,2}"] \& \& \& \C^2 \arrow[from=uu, crossing over, "g_{2,2}"] \arrow[dl, "\iota_{3,2}"]\\
\C^3 \arrow[rrr, "f_3^1"] \& \& \& \C^3 \arrow[from=uu,yshift=2, crossing over, "g_{3,2}"]
 \end{tikzcd}
 \caption{The diagram formed by $M^f$, $N$ and $g$ in the example}
\label{fig:3Ddiagramexample}
\end{figure}
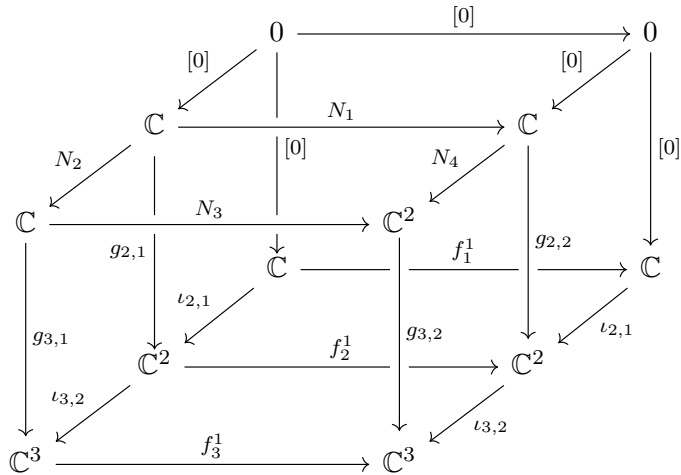

In this case, the group $G_{{\bf e}^w}$ is
\begin{equation*}
      G_{{\bf e}^w}=\prod_{(i,j)\in (\Gamma,I)_0} \GL(\C^{e^w_{i,j}})= \GL(\C) \times \GL(\C) \times \GL(\C) \times \GL(\C^2)
\end{equation*}

and has therefore dimension $\dim(G_{{\bf e}^w})=7$. For a generic dimension vector ${\bf e}^w$, this dimension is given by
\begin{equation*}
    \dim(G_{{\bf e}^w})=\sum_{(i,j)\in (\Gamma,I)_0} (e^w_{i,j})^2.
\end{equation*}

Since $\dim(\Gr_{{\bf e}^w}(M^f))=2,$ the dimension of $\Hom^0({\bf e}^w, M^f)$ is
\begin{equation}\label{eq:dimHom0}
    \dim(\Hom^0({\bf e}^w, M^f))= \dim(\Gr_{{\bf e}^w}(M^f)) + \dim(G_{{\bf e}^w}) = 7+2=9.
\end{equation}

In order to compute the codimension of $\Hom^0({\bf e}^w,M^f)$ in $V$, we first need the dimension of $V$, which is
\begin{equation*}
    \dim(V)=\dim(R_{{\bf e}^w}) + \dim\Bigl(\prod_{(p,q)\in (\Gamma,I)_0} \Hom(\C^{e^w_{p,q}}, \C^i)\Bigr).
\end{equation*}

For the second summand, we have
\begin{equation*}
    \prod_{(p,q)\in (\Gamma,I)_0} \Hom(\C^{e^w_{p,q}}, \C^i)= \Hom(\C,\C^2)\times \Hom(\C,\C^2) \times \Hom(\C,\C^3) \times \Hom(\C^2,\C^3) 
\end{equation*}
and therefore
\begin{equation}\label{eq:dimhomspace}
    \dim \Bigl( \prod_{(p,q)\in (\Gamma,I)_0} \Hom(\C^{e^w_{p,q}}, \C^i)\Bigr)=13.
\end{equation}

To find the dimension of $R_{{\bf e}^w}$, we first observe that it is the subvariety of representations of $(\Gamma,I)$ with dimension vector ${\bf e}^w$ that satisfy the commutativity relations. This means
\begin{equation*}
    R_{{\bf e}^w}\subset \Hom(\C,\C) \times \Hom(\C,\C) \times \Hom(\C,\C^2) \times \Hom(\C,\C^2),
\end{equation*}
that is, $R_{{\bf e}^w}$ is a subvariety of a six-dimensional variety of representations.

In particular, since $e^w_{1,1}=e^w_{1,2}=0$, a representation $N$ in $R_{{\bf e}^w}$ is determined by four linear maps $N_1,N_2,N_3$ and $N_4$ of the following form:
\begin{equation*}
\begin{tikzcd}[]
\overset{\C}{\bullet} \ar[r, "N_1"] \ar[d, "N_2"]  & \overset{\C}{\bullet} \ar[d, "N_4"] \\
\overset{\C}{\bullet} \ar[ur,phantom, "\scalebox{1.5}{$\circlearrowleft$}"] \ar[r, "N_3"] & \overset{\C^2}{\bullet}
\end{tikzcd},
\end{equation*}
and we can write the matrices representing these linear maps as

\begin{equation*}
    N_1=[x_1], N_2=[x_2], N_3=\begin{bsmallmatrix}
        x_3\\
        x_4
    \end{bsmallmatrix}, N_4= \begin{bsmallmatrix}
        x_5\\
        x_6
    \end{bsmallmatrix}.
\end{equation*}

The commutativity relation implies that $N\in R_{{\bf e}^w}$ if and only if the two independent relations $x_3x_2=x_5x_1$ and $x_4x_2=x_6x_1$ are satisfied. Thus, the dimension of $R_{{\bf e}^w}$ is
\begin{equation}\label{eq:dimRe^w}
    \dim(R_{{\bf e}^w})= 6-2 =4.
\end{equation}

Now, we put together \eqref{eq:dimhomspace} and \eqref{eq:dimRe^w} and obtain $\dim(V)=13+4=17$. From this and from \eqref{eq:dimHom0}, it follows that the codimension of $\Hom^0({\bf e}^w,M^f)$ in $V$ is equal to $\dim(V)-\dim(\Hom^0({\bf e}^w,M^f))=17-9=8$.

The space $\Hom^0({\bf e}^w,M^f)$ is locally a complete intersection if its codimension in $V$ is equal to the number of equations that define $\Hom({\bf e}^w,M^f)$ in $V$. A point $((N_{\alpha}),(g_{p,q}))$ is in $\Hom({\bf e}^w,M^f)$ if and only if it satisfies the relations given in \eqref{eq:homvariety}, which, in the case of this example, are
\begin{equation}\label{eq:11relations}
    \begin{cases}
        f_2^1 g_{2,1} = g_{2,2} N_1 \\
        \iota_{3,2} g_{2,1} = g_{3,1} N_2 \\
        f_3^1 g_{3,1} = g_{3,2} N_3 \\
        \iota_{3,2} g_{2,2} = g_{3,2} N_4
    \end{cases}.
\end{equation}

We represent the linear maps $g_{p,q}$ with the following matrices:
\begin{equation*}
    g_{2,1}=\begin{bsmallmatrix}
        y_1\\
        y_2
    \end{bsmallmatrix}, g_{3,1}=\begin{bsmallmatrix}
        y_3\\
        y_4\\
        y_5
    \end{bsmallmatrix}, g_{3,2}=\begin{bsmallmatrix}
        y_6 & y_7\\
        y_8 & y_9\\
        y_{10} & y_{11}
    \end{bsmallmatrix}, g_{2,2}=\begin{bsmallmatrix}
        y_{12}\\
        y_{13}
    \end{bsmallmatrix}
\end{equation*}
and make the relations in \eqref{eq:11relations} explicit, obtaining the following 11 relations:
\begin{equation}\label{eq:11relationsexplicit}
    \begin{cases}
        \begin{bsmallmatrix}
        y_1\\
        y_2
    \end{bsmallmatrix} & = \begin{bsmallmatrix}
        y_{12} x_1\\
        y_{13} x_1
    \end{bsmallmatrix} \\
    \begin{bsmallmatrix}
        y_1\\
        y_2\\
        0
    \end{bsmallmatrix} & = \begin{bsmallmatrix}
        y_3 x_2\\
        y_4 x_2\\
        y_5 x_2
    \end{bsmallmatrix} \\
    \begin{bsmallmatrix}
        y_3\\
        y_4\\
        y_5
    \end{bsmallmatrix} & = \begin{bsmallmatrix}
        y_6 x_3+y_7x_4\\
        y_8 x_3+y_9x_4\\
        y_{10} x_3+y_{11} x_4
    \end{bsmallmatrix} \\
    \begin{bsmallmatrix}
        y_{12}\\
        y_{13}\\
        0
    \end{bsmallmatrix} & = \begin{bsmallmatrix}
        y_6 x_5+y_7 x_6\\
        y_8 x_5+y_9 x_6 \\
        y_{10} x_5+y_{11} x_6
    \end{bsmallmatrix}
    \end{cases}.
\end{equation}

By imposing $x_3x_2=x_5x_1$ and $x_4x_2=x_6x_1$ (the relations that follow from $N$ being a representation of $(\Gamma,I)$) and operating two substitutions in \eqref{eq:11relationsexplicit} we notice that 3 out of those 11 relations can be derived from the other 8.
In particular, a set of independent, generating relations is given either by
\begin{equation*}
    \begin{cases}
        \begin{bsmallmatrix}
        y_3\\
        y_4\\
        y_5
    \end{bsmallmatrix} & = \begin{bsmallmatrix}
        y_6 x_3+y_7x_4\\
        y_8 x_3+y_9x_4\\
        y_{10} x_3+y_{11} x_4
    \end{bsmallmatrix} \\
    \begin{bsmallmatrix}
        y_{12}\\
        y_{13}\\
        0
    \end{bsmallmatrix} & = \begin{bsmallmatrix}
        y_6 x_5+y_7 x_6\\
        y_8 x_5+y_9 x_6 \\
        y_{10} x_5+y_{11} x_6
    \end{bsmallmatrix} \\
    \begin{bsmallmatrix}
        y_1\\
        y_2
    \end{bsmallmatrix} & = \begin{bsmallmatrix}
        y_{12} x_1\\
        y_{13} x_1
    \end{bsmallmatrix}
    \end{cases}
\end{equation*}

or by
\begin{equation*}
    \begin{cases}
        \begin{bsmallmatrix}
        y_3\\
        y_4\\
        y_5
    \end{bsmallmatrix} & = \begin{bsmallmatrix}
        y_6 x_3+y_7x_4\\
        y_8 x_3+y_9x_4\\
        y_{10} x_3+y_{11} x_4
    \end{bsmallmatrix} \\
    \begin{bsmallmatrix}
        y_{12}\\
        y_{13}\\
        0
    \end{bsmallmatrix} & = \begin{bsmallmatrix}
        y_6 x_5+y_7 x_6\\
        y_8 x_5+y_9 x_6 \\
        y_{10} x_5+y_{11} x_6
    \end{bsmallmatrix} \\
    \begin{bsmallmatrix}
        y_1\\
        y_2
    \end{bsmallmatrix} & = \begin{bsmallmatrix}
        y_3 x_2\\
        y_4 x_2
    \end{bsmallmatrix}
    \end{cases}.
\end{equation*}

We found - as expected - that the codimension of $\Hom^0({\bf e}^w,M^f)$ in $V$ is equal to the number of equations that define $\Hom({\bf e}^w,M^f)$ in $V$.
What motivates Conjecture \ref{conj:flatlocus} is the fact that the strategy and methods employed here do not depend on the specific choice of $f$. Therefore, we can expect them to provide the same result for any $f$ such that $\dim(\Gr_{{\bf e}^w}(M^f))=\dim(X_w)$ and in generic ambient dimension $n+1$.

\bibliography{bibliography.bib}
\bibliographystyle{amsplain}

\end{document}